\documentclass{amsart}
\usepackage{amscd,amsmath,xypic,amssymb,combelow,tikz-cd,etoolbox,calligra,mathrsfs,stmaryrd}

\DeclareMathOperator{\sHom}{\mathscr{H}\text{\kern -3pt {\calligra\large om}}\,}
\DeclareMathOperator{\sRHom}{{R}\mathscr{H}\text{\kern -3pt {\calligra\large om}}\,}
\DeclareMathOperator{\sExt}{\mathscr{E}\text{\kern -3pt {\calligra\large xt}}\,}
\DeclareMathOperator{\sQuot}{\mathscr{Q}\text{\kern -3pt {\calligra\large uot}}\,}

\makeatletter
\patchcmd{\@settitle}{\uppercasenonmath\@title}{}{}{}
\makeatother

\xyoption{all}
\CompileMatrices

\newcommand{\nc}{\newcommand}

\makeatletter
\@addtoreset{equation}{section}
\makeatother

\newtheorem{theorem}[subsection]{Theorem}
\newtheorem{proposition}[subsection]{Proposition}
\newtheorem{lemma}[subsection]{Lemma}
\newtheorem{corollary}[subsection]{Corollary}

\newtheorem{definition}[subsection]{Definition}

\newtheorem{claim}[subsection]{Claim}

\newtheorem{example}[subsection]{Example}

\nc{\fa}{{\mathfrak{a}}}
\nc{\fb}{{\mathfrak{b}}}
\nc{\fg}{{\mathfrak{g}}}
\nc{\fh}{{\mathfrak{h}}}
\nc{\fj}{{\mathfrak{j}}}
\nc{\fn}{{\mathfrak{n}}}
\nc{\fm}{{\mathfrak{m}}}
\nc{\fu}{{\mathfrak{u}}}
\nc{\fp}{{\mathfrak{p}}}
\nc{\fr}{{\mathfrak{r}}}
\nc{\ft}{{\mathfrak{t}}}
\nc{\fsl}{{\mathfrak{sl}}}
\nc{\fgl}{{\mathfrak{gl}}}
\nc{\hsl}{{\widehat{\mathfrak{sl}}}}
\nc{\hgl}{{\widehat{\mathfrak{gl}}}}
\nc{\hg}{{\widehat{\mathfrak{g}}}}
\nc{\chg}{{\widehat{\mathfrak{g}}}{}^\vee}
\nc{\hn}{{\widehat{\mathfrak{n}}}}
\nc{\chn}{{\widehat{\mathfrak{n}}}{}^\vee}

\nc{\Mod}{{\textrm{Mod}}}

\nc{\wGL}{{\widehat{GL}^+}}

\nc{\BA}{{\mathbb{A}}}
\nc{\BC}{{\mathbb{C}}}
\nc{\BG}{{\mathbb{G}}}
\nc{\BK}{{\mathbb{K}}}
\nc{\BM}{{\mathbb{M}}}
\nc{\BN}{{\mathbb{N}}}
\nc{\BF}{{\mathbb{F}}}
\nc{\BH}{{\mathbb{H}}}
\nc{\BP}{{\mathbb{P}}}
\nc{\BQ}{{\mathbb{Q}}}
\nc{\BR}{{\mathbb{R}}}
\nc{\BZ}{{\mathbb{Z}}}

\nc{\ff}{{\mathbb{F}}}
\nc{\kk}{{\mathbb{K}}}
\nc{\kko}{{\mathbb{K}}}

\nc{\coh}{{\text{Coh}}}

\nc{\CA}{{\mathcal{A}}}
\nc{\CC}{{\mathcal{C}}}
\nc{\CB}{{\mathcal{B}}}
\nc{\DD}{{\mathcal{D}}}
\nc{\CE}{{\mathcal{E}}}
\nc{\CF}{{\mathcal{F}}}
\nc{\tCF}{{\widetilde{\CF}}}
\nc{\tCM}{{\widetilde{\CM}}}
\nc{\tCT}{{\widetilde{\CT}}}
\nc{\oCF}{{\bar{\CF}}}
\nc{\CG}{{\mathcal{G}}}
\nc{\CL}{{\mathcal{L}}}
\nc{\CK}{{\mathcal{K}}}
\nc{\CI}{{\mathcal{I}}}
\nc{\CM}{{\mathcal{M}}}
\nc{\CH}{{\mathcal{H}}}
\nc{\CN}{{\mathcal{N}}}
\nc{\CO}{{\mathcal{O}}}
\nc{\CP}{{\mathcal{P}}}
\nc{\CR}{{\mathcal{R}}}
\nc{\CQ}{{\mathcal{Q}}}
\nc{\CS}{{\mathcal{S}}}
\nc{\CT}{{\mathcal{T}}}
\nc{\tCU}{{\widetilde{\CU}}}
\nc{\CU}{{\mathcal{U}}}
\nc{\CV}{{\mathcal{V}}}
\nc{\CW}{{\mathcal{W}}}

\nc{\tpsi}{{\widetilde{\Psi}}}
\nc{\wpi}{{\widetilde{\pi}}}
\nc{\Ker}{{\text{Ker }}}

\nc{\CX}{{\mathcal{X}}}
\nc{\tCX}{{\widetilde{\mathcal{X}}}}
\nc{\CY}{{\mathcal{Y}}}
\nc{\tCY}{{\widetilde{\mathcal{Y}}}}

\nc{\tN}{{\widetilde{\CN}}}
\nc{\pN}{{\BP\widetilde{\CN}}}

\nc{\tT}{{T}}

\nc{\fC}{{\mathfrak{C}}}
\nc{\fZ}{{\mathfrak{Z}}}
\nc{\fU}{{\mathfrak{U}}}
\nc{\fV}{{\mathfrak{V}}}
\nc{\fS}{{\mathfrak{S}}}

\nc{\od}{{\overline{d}}}
\nc{\rg}{{\textrm{R}\Gamma}}
\nc{\erg}{{\emph{R}\Gamma}}
\nc{\id}{{\textrm{Id}}}
\nc{\rhom}{{\textrm{RHom}}}

\def\pt{\textrm{pt}}
\def\and{\textrm{ }\&\textrm{ }}

\def\sym{\textrm{Sym}}

\def\proj{\textrm{proj}}

\def\tCF{\widetilde{\CF}}

\def\red{\text{red}}

\def\loccit{\emph{loc. cit. }}

\def\Id{\text{Id}}

\def\sgn{\text{sign }}

\def\km{K_\CM}
\def\kmm{K_{\CM'}}
\def\kms{K_{\CM \times S}}
\def\kmss{K_{\CM \times S \times S}}
\def\kmsss{K_{\CM \times S \times S \times S}}
\def\kmms{K_{\CM' \times S}}

\def\ks{K_S}

\def\fr{\text{Frac }\BF}

\def\pic{\text{Pic}}

\def\vac{{|\varnothing\rangle}}

\begin{document}

\title[AGT relations for sheaves on surfaces]{\large{\textbf{AGT RELATIONS FOR SHEAVES ON SURFACES}}}
\author[Andrei Negu\cb t]{Andrei Negu\cb t}
\address{MIT, Department of Mathematics, Cambridge, MA, USA}
\address{Simion Stoilow Institute of Mathematics, Bucharest, Romania}
\email{andrei.negut@gmail.com}

\maketitle

\renewcommand{\thefootnote}{\fnsymbol{footnote}} 
\renewcommand{\thefootnote}{\arabic{footnote}} 

\renewcommand{\thefootnote}{\fnsymbol{footnote}} 


\renewcommand{\thefootnote}{\arabic{footnote}} 

\begin{abstract} We consider a natural generalization of the Carlsson-Okounkov Ext operator on the $K$--theory groups of the moduli spaces of stable sheaves on a smooth projective surface. We compute the commutation relations between the Ext operator and the action of the deformed $W$--algebra on $K$--theory, which was developed in \cite{Univ, W surf, Hecke}. The conclusion is that the Ext operator is closely related to a vertex operator, thus giving a mathematical incarnation of the Alday-Gaiotto-Tachikawa correspondence for a general algebraic surface.

\end{abstract}

\bigskip

\section{Introduction} 
\label{sec:introduction}

\medskip



\subsection{} Fix a smooth projective surface $S$ over an algebraically closed field of characteristic 0 (henceforth denoted by $\BC$), and invariants $(r,c_1) \in \BN \times H^2(S,\BZ)$. An important object in algebraic geometry is the moduli space:
\begin{equation}
\label{eqn:moduli}
\CM = \bigsqcup^\infty_{c_2 = \left \lceil \frac {r-1}{2r} c_1^2 \right \rceil} \CM_{c_2}
\end{equation}
of $H$--stable sheaves on $S$ with invariants $(r,c_1,c_2)$, for any $c_2 \in \BZ$. The reason why $c_2$ is bounded below is called Bogomolov's inequality, which states that there are no $H$--stable sheaves if $c_2 < \frac {r-1}{2r} c_1^2$. We make the same assumptions as in \cite{Univ, W surf, Hecke}: \\

\begin{itemize}

\item Assumption A: $\gcd(r,c_1 \cdot H) = 1$ \\

\item Assumption S: either $\begin{cases} \omega_S \cong \CO_S, \ \text{ or} \\ c_1(\omega_S) \cdot H < 0 \end{cases}$ \\

\end{itemize}

\noindent Assumption A implies that $\CM$ is proper and there exists a universal sheaf \footnote{We require the universal sheaves on the various connected components of $\CM$ to be constructed as in \cite[Subsection 5.9]{Univ}, which will ensure that they lift in a compatible way to the moduli spaces $\fZ_1$, $\fZ_2^\bullet$ of Subsection \ref{sub:basic mod}.}:
\begin{equation}
\label{eqn:universal}
\xymatrix{\CU \ar@{.>}[d] & \\ 
\CM \times S}
\end{equation}
Assumption S implies that $\CM$ is smooth. \\

\subsection{} The enumerative geometry of the moduli space of stable sheaves is quite rich, as evidenced by Donaldson invariants arising as certain integrals of cohomology classes on $\CM$. In the present paper, we will consider algebraic $K$-theory instead of cohomology, a process which accounts for the adjective ``deformed" in the representation theoretic structures explained in Subsection \ref{sub:intro action}. Explicitly, we consider the following algebraic $K$--theory groups with $\BQ$ coefficients:
\begin{equation}
\label{eqn:km}
\km = \bigoplus^\infty_{c_2 = \left \lceil \frac {r-1}{2r} c_1^2 \right \rceil} K_0(\CM_{c_2}) \underset{\BZ}{\otimes} \BQ
\end{equation}
Let $m \in \pic(S)$, and consider two copies $\CM$ and $\CM'$ of the moduli space \eqref{eqn:moduli}. These two copies may be defined with respect to a different $c_1$ and stability condition $H$, but we assume that the rank $r$ of the sheaves parameterized by $\CM$ and $\CM'$ is the same. In this paper, we will mostly be concerned with the virtual vector bundle:
\begin{equation}
\label{eqn:e}
\xymatrix{ & \CE_m \ar@{-->}[d] \\
& \CM \times \CM' \ar[ld]_{\pi_1} \ar[rd]^{\pi_2} \\ 
\CM & & \CM'}
\end{equation}
(which is a straightforward generalization of the construction of \cite{CO}) given by:
\begin{equation}
\label{eqn:sheaf e}
\CE_m = \text{R}\Gamma(m) - \text{R}\pi_* \left(\sRHom \left(\CU', \CU \otimes m \right)\right)
\end{equation}
The $\sRHom$ is computed on $\CM \times \CM' \times S$: the notation $\CU$, $\CU'$, $m$ stands for the pull-back of the universal sheaf from $\CM \times S$ and $\CM' \times S$, respectively, as well as the pull-back of the line bundle $m$ from $S$. Similarly, $\pi : \CM \times \CM' \times S \rightarrow \CM \times \CM'$ is the standard projection, so $\CE_m$ is a complex of coherent sheaves on $\CM \times \CM'$. \\

\subsection{} Any Schur functor applied to $\CE_m$ gives rise to a $K$-theory class on $\CM \times \CM'$, which in turn induces an operator from $\kmm$ to $\km$ via the usual formalism of correspondences. With this in mind, let us consider the following immediate generalization of \cite[equation (3)]{CO} and \cite[equation (19)]{CNO}. \\

\begin{definition}
\label{def:am}

Consider the so-called \textbf{Ext operator} $\kmm \xrightarrow{A_m} \km$ given by:
\begin{equation}
\label{eqn:a correspondence}
A_m = \pi_{1*} \left(\wedge^\bullet \CE_m  \cdot \pi_2^{*} \right)
\end{equation}
with $\pi_1$ and $\pi_2$ as in \eqref{eqn:e}. The push-forward and pull-back maps are well-defined due to the properness and smoothness of $\CM$ and $\CM'$, respectively. \\
\end{definition}

\noindent In \eqref{eqn:a correspondence}, the symbol $\wedge^\bullet \CE_m$ denotes the total exterior power of $\CE_m$; as $\CE_m$ is in general a complex of coherent sheaves, some explanation is in order. Specifically, consider: 
\begin{equation}
\label{eqn:power series expansion}
\wedge^\bullet \left( \frac {\CE_m}t \right) = \sum_{k \geq 0} (-t)^{-k} [\wedge^k \CE_m] \in K_{\CM \times \CM'}[[t^{-1}]]
\end{equation}
where the right-hand side is the power series expansion of a rational function in $t$ (see Subsection \ref{sub:exterior} for details). Then the quantity $\wedge^\bullet \CE_m$ in \eqref{eqn:a correspondence} denotes the $t=1$ specialization of \eqref{eqn:power series expansion}. If this specialization is not well-defined, then all the results in Subsections \ref{sub:intro action} and \ref{sub:intro final} hold with $m$ replaced by $\frac mt$, and all formulas being equalities of rational functions in $t$; see Subsection \ref{sub:exterior} for details. \\

\begin{example}
\label{ex:y genus}

Let $\CM = \CM'$ and $m = \frac {\CO_S}t$, with $t$ being a formal parameter. Then Assumption S implies that $\CE_{\CO_S}$ is locally free (up to a constant sheaf) and that:
$$
\CE_{\CO_S} \Big|_{\Delta} \cong \emph{Tan}_{\CM}
$$
where $\Delta \subset \CM \times \CM'$ denotes the diagonal. By a simple computation involving correspondences, the isomorphism above implies that:
$$
\emph{Tr}(A_{\frac {\CO_S}t}) = \sum_{k \geq 0} (-t)^{-k} \chi(\CM, \wedge^k \emph{Tan}_{\CM})
$$
(up to a constant rational function in $t$). The right-hand side is the $\chi_t$-genus of the moduli space $\CM$, as considered for example in \cite{GK}. \\

\end{example}

\subsection{} 
\label{sub:intro action}

In the present paper, we will seek to determine the Ext operator $A_m$ using the representation theoretic properties of the vector space $\km$. To this end, we need to make $\km$ into a representation of an appropriate algebra, which is ``big" enough in order to constrain the operator $A_m$ as much as possible. A candidate for such an algebra is $\CA_r$, namely a particular integral form of the deformed $W$--algebra of type $\fgl_r$ (initially defined in \cite{AKOS, FF}). The main purpose of \cite{Univ, W surf, Hecke} is to construct an action $\CA_r \curvearrowright \km$; we will recall the construction in Section \ref{sec:mod}, but let us summarize the main idea here. In \cite[Section 6.7]{W surf}, we construct certain geometric operators:
\begin{equation}
\label{eqn:def w}
\km \xrightarrow{W_{n,k}} \kms
\end{equation}
$\forall (n,k) \in \BZ \times \BN$. Under Assumptions A and S, we show in \cite[Theorem 4.15]{Hecke} that the operators $W_{n,k}$ satisfy the quadratic commutation relations developed in \cite{AKOS, FF} (see \eqref{eqn:comm w} for the specific form of these relations in our language). In \cite[Theorem 6.9]{W surf}, we further show that $W_{n,k} = 0$ for all $n \in \BZ$ and $k > r$, which tautologically implies that the operators \eqref{eqn:def w} yield an action $\CA_r \curvearrowright \km$. Write:
\begin{equation}
\label{eqn:def q}
q = [\omega_S] \in K_S := K_0(S) \otimes_{\BZ} \BQ
\end{equation}
Given two copies $\CM$ and $\CM'$ of the moduli space of stable sheaves, each with its own universal sheaf $\CU$ and $\CU'$, respectively, we may write:
\begin{equation}
\label{eqn:def u}
u = \det \CU \qquad \text{and} \qquad u' = \det \CU'
\end{equation}
for the determinant line bundles on $\CM \times S$ and $\CM' \times S$, respectively. We set:
\begin{equation}
\label{eqn:def gamma}
\gamma = \frac {m^r u}{q^r u'}
\end{equation}
which is the class of a line bundle on $\CM \times \CM' \times S$ (it is implicit that $m$ and $q$ are pulled back from $S$). Our main result, which will be proved in Section \ref{sec:ext}, is: \\

\begin{theorem}
\label{thm:main}

We have the following interaction between the Ext operator \eqref{eqn:a correspondence} and the generators \eqref{eqn:def w} of the $W$--algebra action:
\begin{equation}
\label{eqn:comm a}
A_m W_k(x) (1-x)  = m^k W_k(x\gamma) A_m \left(1- \frac x{q^k} \right)
\end{equation}
where $W_k(x) = \sum_{n \in \BZ} \frac {W_{n,k}}{x^{n}}$. The series coefficients of the two sides of \eqref{eqn:comm a} are maps $\kmm \rightarrow \kms$ which arise from certain correspondences $\in K_{\CM \times \CM' \times S}$ \footnote{See Subsection \ref{sub:correspondences} for a review of correspondences as $K$--theoretic operators. In particular, the composition of operators depends on which of $A_m$ and $W_k(x)$ is on the left of the other:
\begin{align*}
&A_m W_{n,k} : \kmm \xrightarrow{W_{n,k}} \kmms \xrightarrow{A_m \times \text{Id}_S} \kms \\ 
&W_{n,k} A_m : \kmm \xrightarrow{A_m} \km \xrightarrow{W_{n,k}} \kms
\end{align*}
The expressions above are actually given by certain correspondences in $K_{\CM \times \CM' \times S}$. Then the factors $q$ and $\gamma$ in the right-hand side of \eqref{eqn:comm a} indicate multiplication of the aforementioned  correspondences by various powers of the line bundles \eqref{eqn:def q} and \eqref{eqn:def gamma}.}. \\

\end{theorem}

\subsection{} A major motivation for the study of the Ext operator $A_m$ stems from mathematical physics: as explained in \cite{CNO}, the operator $A_m$ encodes the contribution of bifundamental matter to partition functions of 5d supersymmetric gauge theory on the algebraic surface $S$ times a circle. Moreover, the deformed $W$--algebra $\CA_r$ encodes symmetries of Toda conformal field theory. In this language, \eqref{eqn:comm a} becomes a mathematical manifestation of the Alday-Gaiotto-Tachikawa (AGT) correspondence between gauge theory and conformal field theory, by describing the Ext operator $A_m$ in terms of its commutation with $W$--algebra generators. To the author's knowledge, the present paper is the first mathematical treatment of AGT over general algebraic surfaces in rank $r>1$ (the reference \cite{CNO} used different techniques from ours to describe the Ext operator in the $r=1$ case). \\

\noindent However, we note that formulas \eqref{eqn:comm a} are not enough to completely determine $A_m$ for a general smooth projective surface $S$, and one should instead work with a deformed vertex operator algebra which properly contains several deformed $W$--algebras $\CA_r$. In the non-deformed case, a potential candidate for such a larger algebra was studied in \cite{FG}, where the authors expect that it contains operators which modify sheaves on $S$ along entire curves, on top of our operators $W_{n,k}$ which modify sheaves at individual points. While we give a complete algebro-geometric description of the latter operators, we do not have such a description for the former operators. Once such a description will be available, we hope that one can extend Theorem \ref{thm:main} to a bigger vertex operator algebra properly containing $\CA_r$. \\

\noindent There is a situation where formulas \eqref{eqn:comm a} do indeed determine the Ext operator $A_m$ completely: this corresponds to taking $S = \BA^2$, replacing $\CM$ by the moduli space of framed rank $r$ sheaves on the projective plane, and working with torus equivariant $K$-theory (see Subsection \ref{sub:verma intro} for details). In this particular case, we showed in \cite{W} that $\km$ is isomorphic to the universal Verma module of $\CA_r$. Theorem \ref{thm:main} holds in the situation at hand, and we will show in Theorem \ref{thm:unique} that our formulas completely determine the Ext operator $A_m$. This precisely yields the AGT correspondence for 5d supersymmetric gauge theory on $\BA^2 \times S^1$ (see e.g. \cite{BFN, BPSS, MO, SV} for the history of this correspondence in mathematical language). \\

\subsection{} 
\label{sub:intro final} 

Alongside the operators \eqref{eqn:def w}, we constructed in \cite[Theorem 4.15]{Hecke} $K$--theory lifts of the operators introduced by Grojnowski and Nakajima (\cite{G, Nak}, for $r=1$) and generalized by Baranovsky (\cite{Ba}, for any $r$) in cohomology:
\begin{equation}
\label{eqn:def heis}
\km \xrightarrow{P_n} \kms
\end{equation}
$\forall n \in \BZ \backslash 0$. These operators satisfy the Heisenberg commutation relation \eqref{eqn:q heis}, and interact with the deformed $W$--algebra generators according to relation \eqref{eqn:w and p}. \\

\noindent Recall the line bundles $q$ and $\gamma$ of \eqref{eqn:def q} and \eqref{eqn:def gamma}, respectively, and the footnote in Theorem \ref{thm:main} to properly interpret compositions of the operators $A_m$ and $P_{\pm n}$. \\

\begin{theorem}
\label{thm:main heis}

We have the following interaction between the Ext operator \eqref{eqn:a correspondence} and the Heisenberg operators $P_{\pm n}$, for all $n>0$:
\begin{equation}
\label{eqn:comm a heis 1}
A_m P_{-n} - P_{-n} A_m \gamma^n = A_m(1-\gamma^n)
\end{equation} 
\begin{equation}
\label{eqn:comm a heis 2}
\ \ A_m P_n - P_n A_m \gamma^{-n} = A_m(\gamma^{-n} - q^{rn}) 
\end{equation}

\medskip

\end{theorem}

\noindent In $\CA_r$, the series $W_r(x)$ matches the normal-ordered exponential of the generating series of the $P_n$'s (see Theorem \ref{thm:w}). With this in mind, it is straightforward to show that the $k=r$ case of Theorem \ref{thm:main} follows from Theorem \ref{thm:main heis}. \\


\noindent For any $\alpha \in K_S$, we will write $P_n\{\alpha\}$ for the composition: 
$$
P_n\{\alpha\} : \km \xrightarrow{P_n} \kms \xrightarrow{\text{multiplication by } \proj_2^*(\alpha)} \kms \xrightarrow{\proj_{1*}} \km
$$
where $\proj_1$, $\proj_2$ are the projections from $\CM \times S$ to $\CM$ and $S$, respectively. Let $q_1$ and $q_2$ denote the Chern roots of the cotangent bundle $\Omega_S^1$. Any symmetric Laurent polynomial in $q_1$, $q_2$ gives rise to a well-defined element of $\ks$, via:
$$
q_1+q_2 = [\Omega^1_S] \quad \text{and} \quad q = q_1q_2 = [\omega_S]
$$
Define:
\begin{equation}
\label{eqn:phi}
\Phi_m = A_m \exp \left[\sum_{n=1}^\infty \frac {P_n}n \left\{  \frac {(q^n - 1)q^{-nr}}{[n]_{q_1} [n]_{q_2}} \right\} \right]
\end{equation}
where $[n]_x = 1+x+\dots+x^{n-1}$. The expression in curly brackets is an element of $\ks$ because $[n]_{q_1} [n]_{q_2}$ is a unit in the ring $K_S$ \footnote{Indeed, since the Chern character gives us an isomorphism $\ks \cong A^*(S,\BQ)$, then $q_1+q_2 = [\Omega^1_S] \in 2 + \mathcal{N}$ and $q = [\omega_S] \in 1 + \mathcal{N}$, where $\mathcal{N} \subset \ks$ denotes the nilradical. Therefore, $[n]_{q_1} [n]_{q_2} \in n^2 + \mathcal{N}$, and is thus invertible in the ring $K_S$.}. \\

\begin{corollary}
\label{cor:main}

Formulas \eqref{eqn:comm a}, \eqref{eqn:comm a heis 1}, \eqref{eqn:comm a heis 2} imply the following:
\begin{equation}
\label{eqn:comm phi}
\Big[ \Phi_m W_k(x) - m^k W_k(x\gamma) \Phi_m \Big] \left(1- \frac x{q^k} \right) = 0
\end{equation}
\begin{equation}
\label{eqn:comm phi heis}
\ \Phi_m P_{\pm n} - P_{\pm n} \Phi_m \gamma^{\mp n} = \pm \Phi_m(\gamma^{\mp n} - q^{\pm rn})
\end{equation}
$\forall k,n > 0$. An operator $\Phi_m$ satisfying \eqref{eqn:comm phi}, \eqref{eqn:comm phi heis} is called a \textbf{vertex operator}. \\

\end{corollary}

\noindent I would like to thank Boris Feigin, Sergei Gukov, Hiraku Nakajima, Nikita Nekrasov, Andrei Okounkov, Francesco Sala and Alexander Tsymbaliuk for many interesting discussions on the subject of Ext operators and $W$--algebras. I gratefully acknowledge the support of NSF grant DMS--1600375. \\

\section{The moduli space of sheaves}
\label{sec:mod}

\medskip

\subsection{}
\label{sub:correspondences}

Throughout the present paper, we will work with smooth projective varieties over the field $\BC$. For such varieties $X$, we let:
$$
K_X = K_0(X) \otimes_{\BZ} \BQ
$$
be the Grothendieck group of the category of coherent sheaves on $X$, with scalars extended to $\BQ$. Derived tensor product yields a ring structure on $K_X$, and we have pull-back and push-forward maps for any proper l.c.i. morphism $X \rightarrow Y$. \\

\begin{definition}

Given smooth projective varieties $X$ and $Y$, any class $\Gamma \in K_{X \times Y}$ (called a ``correspondence" in this setup) defines an operator:
\begin{equation}
\label{eqn:correspondence}
K_Y \xrightarrow{\Psi_\Gamma} K_X, \qquad \Psi_\Gamma = \emph{proj}_{X*} \left( \Gamma \cdot \emph{proj}_Y^* \right) 
\end{equation}
where $\emph{proj}_X, \emph{proj}_Y$ denote the projection maps from $X \times Y$ to $X$ and $Y$, respectively. \\

\end{definition}

\noindent The composition of operators \eqref{eqn:correspondence} can also be described as a correspondence:
\begin{equation}
\label{eqn:composition corr}
\Psi_{\Gamma} \circ \Psi_{\Gamma'} = \Psi_{\Gamma''} : K_Z \longrightarrow K_X
\end{equation}
for any $\Gamma \in K_{X \times Y}$ and $\Gamma' \in K_{Y \times Z}$, where:
\begin{equation}
\label{eqn:composition corr 2}
\Gamma'' = \proj_{X \times Z*}(\proj_{X \times Y}^*(\Gamma) \otimes \proj_{Y \times Z}^*(\Gamma'))
\end{equation}
where $\proj_{X \times Y}$, $\proj_{Y \times Z}$, $\proj_{X \times Z}$ are the standard projections from $X \times Y \times Z$ to $X \times Y$, $Y \times Z$, $X \times Z$. Throughout the present paper, all operators $K_Y \rightarrow K_X$ arise from explicit correspondences. While we will use the language of composition of operators for convenience, what is really happening behind the scenes is composition of correspondences under the operation $(\Gamma , \Gamma') \mapsto \Gamma''$ of \eqref{eqn:composition corr 2}. \\


\subsection{} 
\label{sub:action}

In Subsection \ref{sub:intro action}, we referred to various operators $\km \rightarrow \kms$ as defining an action of a certain algebra on $\km$, and we will now explain the meaning of this notion. Given two arbitrary homomorphisms (of abelian groups):
\begin{equation}
\label{eqn:operators}
\km \xrightarrow{x,y} \kms 
\end{equation}
their ``product" $x y |_\Delta$ is defined as the composition:
$$
x y \Big|_\Delta : \km \xrightarrow{y} \kms \xrightarrow{x \times \text{Id}_S} \kmss \xrightarrow{\text{Id}_{\CM} \times \Delta^*} \kms 
$$
where $S \xrightarrow{\Delta} S \times S$ is the diagonal. It is easy to check that $(xy|_\Delta) z|_\Delta = x (yz|_\Delta)|_\Delta$, hence the aforementioned notion of product is associative, and it makes sense to define $x_1\dots x_n|_\Delta$ for arbitrarily many operators $x_1,\dots,x_n : \km \rightarrow \kms$. \\

\noindent Similarly, given two operators \eqref{eqn:operators}, we may define their commutator:
$$
\km \xrightarrow{[x,y]} \kmss 
$$
as the difference of the two compositions:
\begin{align*}
&\km \xrightarrow{y} \kms \xrightarrow{x \times \text{Id}_S} \kmss \\
&\km \xrightarrow{x} \kms \xrightarrow{y \times \text{Id}_S} \kmss \xrightarrow{\text{Id}_{\CM} \times \text{swap}^*} \kmss
\end{align*}
where $\text{swap} :  S \times S \rightarrow S \times S$ is the permutation of the two factors. In all cases studied in the present paper, we will have (below, we abuse notation by writing $\Delta_*$ instead of $(\text{Id}_{\CM} \times \Delta)_*$ for the diagonal map $\kms \rightarrow \kmss$):
$$
[x,y] = \Delta_*(z)
$$
for some $\km \xrightarrow{z} \kms$ which is uniquely determined (the diagonal embedding $\Delta_*$ is injective because it has a left inverse), and which will be denoted by $z = [x,y]_\red$. We leave it as an exercise to the interested reader to prove that the commutator satisfies the Leibniz rule in the form $[xy|_\Delta,z]_\red = x[y,z]_\red|_\Delta + [x,z]_\red y|_\Delta$, and the Jacobi identity in the form $[[x,y]_\red,z]_\red + [[y,z]_\red,x]_\red + [[z,x]_\red,y]_\red = 0$. \\

\noindent Finally, we consider the ring homomorphism $\BK = \BZ[q_1^{\pm 1}, q_2^{\pm 1}]^{\sym} \rightarrow \ks$ given by sending $q_1$ and $q_2$ to the Chern roots of the cotangent bundle of $S$ (therefore, $q=q_1q_2$ goes to the class of the canonical line bundle). We will often abuse notation, and write $q_1,q_2,q$ for the images of the indeterminates in the ring $\ks$. For any $\lambda \in \BK$ and any operator \eqref{eqn:operators}, we may define their product as the composition:
$$
\lambda \cdot x : \km \xrightarrow{x} \kms \xrightarrow{\text{Id}_\CM \times (\text{multiplication by }\lambda)} \kms 
$$
where we identify $\lambda \in \BK$ with its image in $\ks$. With this in mind, the ring $\ks$ can be thought of as the ``ring of constants" for the algebra of operators \eqref{eqn:operators}. \\

\subsection{} 
\label{sub:basic mod} 

Recall the universal sheaf \eqref{eqn:universal}, and consider the derived scheme:
\begin{equation}
\label{eqn:z1}
\fZ_1 = \BP_{\CM \times S} (\CU) \longrightarrow \CM \times S
\end{equation}
Since $\CU$ is isomorphic to a quotient $\CV/\CW$ of vector bundles on $\CM \times S$ (see \cite[Proposition 2.2]{Univ}), the projectivization in \eqref{eqn:z1} is defined as the derived zero locus of a section of a vector bundle on the projective bundle $\BP_{\CM \times S}(\CV)$. However, it was shown in \cite[Proposition 2.10]{Univ} that under Assumption S, the derived zero locus is actually a smooth scheme:
$$
\fZ_1 = \bigsqcup_{c = \left  \lceil \frac {r-1}{2r} c_1^2 \right \rceil}^\infty \fZ_{c+1,c}
$$
whose connected components are given by:
\begin{equation}
\label{eqn:z1 closed}
\fZ_{c+1,c} = \Big \{ (\CF_{c+1}, \CF_c) \text{ s.t. } \CF_{c+1} \subset_x \CF_c \text{ for some } x \in S \Big\} \subset \CM_{c+1} \times \CM_c
\end{equation}
and $\CF' \subset_x \CF$ means that $\CF' \subset \CF$ and the quotient $\CF/\CF'$ is isomorphic to the length 1 skyscraper sheaf at the point $x \in S$. This scheme comes with projection maps:
\begin{equation}
\label{eqn:z1 projections}
\xymatrix{ & \fZ_{c+1,c} \ar[ld]_{p_+} \ar[d]^{p_S} \ar[rd]^{p_-} & \\ 
\CM_{c+1} & S & \CM_c}
\end{equation}
More generally, we defined a derived scheme $\fZ_2^\bullet$ in \cite[Definition 4.17]{W surf}, which was shown (under Assumption S, in \cite[Proposition 4.21]{W surf}) to be a smooth scheme:
$$
\fZ_2^\bullet = \bigsqcup_{c = \left  \lceil \frac {r-1}{2r} c_1^2 \right \rceil}^\infty \fZ_{c+2,c}^\bullet
$$
whose connected components are given by:
\begin{equation}
\label{eqn:z2 closed}
\fZ_{c+2,c}^\bullet = \Big \{ (\CF_{c+2} \subset_x \CF_{c+1} \subset_x \CF_c) \text{ for some } x \in S \Big\} \subset  \CM_{c+2} \times \CM_{c+1} \times \CM_c
\end{equation}
This scheme is equipped with projection maps as in \eqref{eqn:z2 projections} below, but we observe that the rhombus is not derived Cartesian (and this is key to our construction):
\begin{equation}
\label{eqn:z2 projections}
\xymatrix{ & \fZ_{c+2,c}^\bullet \ar[ld]_{\pi_+} \ar[rd]^{\pi_-} & \\ 
\fZ_{c+2,c+1} \ar[rd]_{p_- \times p_S} & & \fZ_{c+1,c} \ar[ld]^{p_+ \times p_S} \\
& \CM_{c+1} \times S&}
\end{equation}
Note that all of the maps in the diagram above are proper, l.c.i. morphisms. Define:
\begin{equation}
\label{eqn:zn components}
\fZ_n^\bullet = \bigsqcup_{c = \left  \lceil \frac {r-1}{2r} c_1^2 \right \rceil}^\infty \fZ_{c+n,c}^\bullet
\end{equation}
whose connected components are given by derived fiber products:
\begin{equation}
\label{eqn:zn}
\fZ_{c+n,c}^\bullet = \fZ_{c+n,c+n-2}^\bullet \underset{\fZ_{c+n-1,c+n-2}}\times \dots \underset{\fZ_{c+2,c+1}}\times \fZ_{c+2,c}^\bullet \rightarrow \CM_{c+n} \times \dots \times \CM_{c}
\end{equation}
While $\fZ_n^\bullet$ is a derived scheme, we note that its closed points are all of the form:
\begin{equation}
\label{eqn:zn points}
\fZ_{c+n,c}^\bullet = \{(\CF_{c+n},\dots,\CF_c) \text{ sheaves s.t. } \CF_{c+n} \subset_x \dots \subset_x \CF_c \text{ for some } x \in S\}
\end{equation}
Therefore, we have the following projection maps, which only remember the smallest and the largest sheaf in a flag \eqref{eqn:zn points} (the notation below generalizes \eqref{eqn:z1 projections}):
\begin{equation}
\label{eqn:zn projections}
\xymatrix{ & \fZ_{c+n,c}^\bullet \ar[ld]_{p_+} \ar[d]^{p_S} \ar[rd]^{p_-} & \\ 
\CM_{c+n} & S & \CM_c}
\end{equation}
In diagram \eqref{eqn:zn projections}, the maps $p_\pm$ are l.c.i. morphisms, and the maps $p_\pm \times p_S$ are proper (they inherit these properties from the maps in \eqref{eqn:z2 projections}). Finally, we consider the line bundles $\CL_1,\dots,\CL_n$ on $\fZ_n^\bullet$, whose fibers are given by:
\begin{equation}
\label{eqn:line bundles}
\CL_i |_{(\CF_{c+n},\dots,\CF_c)} = \CF_{c+n-i,x}/\CF_{c+n-i+1,x}
\end{equation}
on the connected component $\fZ_{c+n,c}^\bullet \subset \fZ_n^\bullet$.  \\

\subsection{} 
\label{sub:w action}

Using the derived scheme \eqref{eqn:zn} and the maps \eqref{eqn:zn projections}, define for all $n,k \in \BN$:
\begin{align}
&K_{\CM} \xrightarrow{L_{n,k}} K_{\CM \times S}, & &L_{n,k} = (-1)^{k-1} (p_+ \times p_S)_* \left(\CL_n^k \cdot p_-^* \right) \label{eqn:op l} \\
&K_{\CM} \xrightarrow{U_{n,k}} K_{\CM \times S}, & &U_{n,k} = \frac{ (-1)^{rn + k-1} u^n}{q^{(r-1)n}} (p_- \times p_S)_* \left( \frac {\CL_n^k}{\CQ^r} \cdot p_+^* \right)  \label{eqn:op u}
\end{align}
where $\CQ = \CL_1\dots\CL_n$, and $u$ is the determinant of the universal sheaf on $\CM \times S$, as in \eqref{eqn:def u}.\footnote{Note that $u$ parameterizes the determinant of any one of the sheaves $\CF_{c+n},\dots,\CF_c$ in a flag \eqref{eqn:zn points}, since these sheaves have canonically isomorphic determinants, see Proposition \ref{prop:doesn't matter}} Implicit in the definitions \eqref{eqn:op l} and \eqref{eqn:op u} is that we define the operators therein for all components $\CM_c$ of the moduli space $\CM$. We also set:
\begin{equation}
\label{eqn:convention}
L_{n,0} = U_{n,0} = \delta_n^0 \qquad \text{and} \qquad L_{0,k} = U_{0,k} = \delta_k^0
\end{equation}
Finally, consider for all $k \in \BN \sqcup 0$ the operators:
\begin{equation}
\label{eqn:op e}
E_k : \km \xrightarrow{\text{pull-back}} \kms \xrightarrow {\text{multiplication by }\wedge^k\CU}\kms
\end{equation}
Since $\CU \cong \CV/\CW$ is a coherent sheaf of projective dimension 1 on $\CM \times S$ (see \cite[Proposition 2.2]{Univ}), the class $\wedge^k \CU$ in \eqref{eqn:op e} is defined by setting:
\begin{equation}
\label{eqn:first ext}
\wedge^\bullet \left( \frac {\CU}z \right) = \frac {\wedge^\bullet \left( \frac {\CV}z \right)}{\wedge^\bullet \left( \frac {\CW}z \right)}
\end{equation}
and picking out the coefficient of $z^{-k}$ when expanding in negative powers of $z$. The reason for our notation of the operators \eqref{eqn:op l}, \eqref{eqn:op u}, \eqref{eqn:op e} is that these three operators are respectively lower triangular, upper triangular, and diagonal with respect to the grading on $\km$ by the second Chern class (see \eqref{eqn:km}). \\

\begin{definition}
\label{def:w}

(\cite[Section 6.7]{W surf}) For any $(n,k) \in \BZ \times \BN$, consider the operators:
\begin{equation}
\label{eqn:op w}
W_{n,k} = \sum^{n_2 - n_1 = n}_{k_0+k_1+k_2 = k} q^{(k-1)n_2} \cdot L_{n_1,k_1} E_{k_0} U_{n_2,k_2} \Big|_\Delta 
\end{equation}
as $k_0,k_1,k_2,n_1,n_2$ run over $\BN \sqcup 0$ (recall the convention \eqref{eqn:convention}). \\

\end{definition}

\noindent Note that \eqref{eqn:op w} is an infinite sum, but its action on $\km$ is well-defined because the operators $L_{n,k}$ (respectively $U_{n,k}$) increase (respectively decrease) the $c_2$ of stable sheaves by $n$, and Bogomolov's inequality ensures that the moduli space of stable sheaves is empty if $c_2$ is small enough. \\

\subsection{}

Similarly with \eqref{eqn:op l} and \eqref{eqn:op u}, we have the following operators for all $n \in \BN$:
\begin{align}
&K_{\CM} \xrightarrow{P_{-n}} K_{\CM \times S}, & &P_{-n} = (p_+ \times p_S)_* \left(\sum_{i=0}^{n-1} \frac {q^{i}  \CL_n}{\CL_{n-i}}  \cdot p_-^* \right) \label{eqn:op p 1} \\
&K_{\CM} \xrightarrow{H_{-n}} K_{\CM \times S}, & &H_{-n} = (p_+ \times p_S)_* \left( p_-^* \right) \label{eqn:op h 1} \\
&K_{\CM} \xrightarrow{P_n} K_{\CM \times S}, & &P_n = (-1)^{rn} u^{n} (p_- \times p_S)_* \left( \sum_{i=0}^{n-1} \frac {q^{i}  \CL_n}{\CQ^r \CL_{n-i}} \cdot p_+^* \right)  \label{eqn:op p 2} \\
&K_{\CM} \xrightarrow{H_n} K_{\CM \times S}, & &H_n = (-1)^{rn} u^{n} (p_- \times p_S)_* \left( \CQ^{-r} \cdot p_+^* \right)  \label{eqn:op h 2}
\end{align}
As a consequence of \cite[formulas (2.15) and (5.18)--(5.21)]{W surf}, the operators $H_{\pm n}$ are to the operators $P_{\pm n}$ as complete symmetric functions are to power sum functions:
\begin{equation}
\label{eqn:h to p}
\sum_{n=0}^\infty \frac {H_{\pm n}}{z^{\pm n}} = \exp \left(\sum_{n=1}^\infty \frac {P_{\pm n}}{nz^{\pm n}} \right) \Big|_\Delta
\end{equation}
or, explicitly, $H_0 = \text{proj}_1^*$ (where $\text{proj}_1:\CM \times S \rightarrow \CM$ is the usual projection) and:
\begin{align*}
&H_{\pm 1} = P_{\pm 1} \\
&H_{\pm 2} = \frac {P_{\pm 1} P_{\pm 1} |_\Delta + P_{\pm 2}}2 \\
&H_{\pm 3} = \frac {P_{\pm 1} P_{\pm 1} P_{\pm 1}|_\Delta + 3P_{\pm 1} P_{\pm 2} |_\Delta + 2 P_{\pm 3}}6 
\end{align*}
and so on. \\

\begin{theorem}
\label{thm:w}

(\cite[Theorem 6.9]{W surf}) The operators \eqref{eqn:op w} satisfy, for all $n \in \BZ$:
\begin{equation}
\label{eqn:w=r}
W_{n,r} = u \sum_{n_1, n_2 \geq 0}^{n_2 - n_1 = n} H_{-n_1} H_{n_2} \Big|_{\Delta}  
\end{equation}
and:
\begin{equation}
\label{eqn:w>r}
W_{n,k} = 0
\end{equation}
for all $k > r$. \\

\end{theorem}

\subsection{} We will now present the interaction of the operators \eqref{eqn:op w}, \eqref{eqn:op p 1}, \eqref{eqn:op p 2}. Recall the commutator construction from Subsection \ref{sub:action}. \\

\begin{theorem}
\label{thm:acts}

(stated in \cite[Theorem 3.13 and Proposition 3.15]{W surf} and proved in \cite[Theorem 4.15]{Hecke} under Assumption S) We have the following formulas for all $n,n' \in \BZ$ and $k,k' \in \BN$:
\begin{align}
&[W_{n,k}, W_{n',k'}] = \Delta_* \left( \mathop{\sum_{k+k' = l+l'}^{\min(l,l') \leq \min(k,k')}}_{m+m' = n+n'}^{\frac ml \leq \frac {m'}{l'}}  c_{n,n',k,k'}^{m,m',l,l'} (q_1,q_2) \cdot W_{m,l} W_{m',l'} \Big|_\Delta \right) \label{eqn:comm w} \\
&[P_{n'}, P_n] = \Delta_* \begin{cases} 0 & \text{if } \sgn(n) = \sgn(n') \\
\delta_{n+n'}^0 n [n]_{q_1} [n]_{q_2} [r]_{q^n} \cdot \emph{proj}_{\CM}^* & \text{if }n'<0<n \end{cases} \label{eqn:q heis} \\ \nonumber \\
&[W_{n',k'}, P_{\pm n}] = \Delta_* \Big( \pm [n]_{q_1} [n]_{q_2} [k']_{q^n} q^{n(r-k') \delta_\pm^+} \cdot W_{\pm n + n',k'}  \Big) \label{eqn:w and p}
\end{align}
where the coefficients $c_{n,n',k,k'}^{m,m',l,l'} (q_1,q_2) \in \ks$ were computed algorithmically in \cite{W surf}. They are certain universal symmetric Laurent polynomials in $q_1,q_2$. \\

\end{theorem}


\noindent Indeed, we show in \cite[Theorem 3.13]{W surf} that \eqref{eqn:comm w} is equivalent to the defining relation in the deformed $W$--algebra $\CA_r$ (with $\Delta_*$ replaced by $(1-q_1)(1-q_2)$). Similarly, relation \eqref{eqn:q heis} is the defining relation in the deformed Heisenberg algebra. As we explained in \cite[Definition 5.2 and formulas (5.20)--(5.21)]{W surf} and proved in \cite[Theorem 4.15]{Hecke}, the fact that the operators \eqref{eqn:op w}, \eqref{eqn:op p 1}, \eqref{eqn:op p 2} satisfy the relations in Theorem \ref{thm:acts} is precisely what we mean when we say that the deformed $W$--algebra $\CA_r$ and the deformed Heisenberg algebra act on the groups $\km$. \\

\subsection{} 
\label{sub:series}

Let us consider the operators of Subsection \ref{sub:w action} and form the generating series:
\begin{equation}
\label{eqn:series}
L_n(y) = \sum_{k=1}^\infty \frac {L_{n,k}}{(-y)^k}, \qquad U_n(y) = \sum_{k=1}^\infty \frac {U_{n,k}}{(-y)^k}
\end{equation}
In other words, these power series are considered as operators:
\begin{align*}
&K_{\CM} \xrightarrow{L_n(y)} K_{\CM \times S} \left\llbracket \frac 1y \right\rrbracket, & &L_n(y) = (p_+ \times p_S)_* \left( \frac 1{1 - \frac {y}{\CL_n}} \cdot p_-^* \right) \\
&K_{\CM} \xrightarrow{U_n(y)} K_{\CM \times S} \left\llbracket \frac 1y \right\rrbracket, & &U_n(y) = \frac {(-1)^{rn} u^n}{q^{(r-1)n}} (p_- \times p_S)_* \left( \frac {\CQ^{-r}}{1 - \frac {y}{\CL_n}} \cdot p_+^* \right)
\end{align*}
We will also consider the operators:
$$
E(y) : \km \xrightarrow{\text{pull-back}} \kms \xrightarrow {\text{multiplication by }\wedge^\bullet \left( \frac {\CU}y \right)}\kms\left\llbracket \frac 1y \right\rrbracket
$$
Furthermore, we will consider the generating series:
\begin{equation}
\label{eqn:big series}
L(x,y) = 1 + \sum_{n=1}^\infty L_n(y)x^{n} \qquad \qquad U(x,y) = 1 + \sum_{n=1}^\infty \frac {U_n(y)}{x^n}
\end{equation}
and also set:
\begin{align} 
&W_k(x) = \sum_{n = -\infty}^\infty \frac {W_{n,k}}{x^n} \label{eqn:series w} \\
&W(x,y) = 1 + \sum_{k=1}^\infty \frac {W_k(x)}{y^k} \label{eqn:big series w}
\end{align}
The definition of the $W$--algebra generators in \eqref{eqn:op w} is equivalent to the following:
\begin{equation}
\label{eqn:op w series}
W\left(x, yD_x \right) = L \left(x, yD_x \right) E\left( yD_x \right) U \left(xq, yD_x \right) \Big|_\Delta
\end{equation}
where $D_x$ is the $q$-difference operator in the variable $x$, i.e. $D_x(f(x)) = f(xq)$. In formula \eqref{eqn:op w series}, we place all powers of $D_x$ to the right (respectively to the left) of all powers of $x$ when writing down the power series $L(x,yD_x)$ (respectively $U(xq,yD_x)$). In terms of generating series, formula \eqref{eqn:w and p} reads:
\begin{equation}
\label{eqn:w to p series}
[W_k(x), P_{\pm n}] = \Delta_* \Big(\pm [n]_{q_1} [n]_{q_2} [k]_{q^n} q^{n(r-k) \delta_{\pm}^+} \cdot x^{\pm n} W_k(x) \Big)
\end{equation}

\medskip

\subsection{} 

Given a rational function $F(z)$, whose set of simple poles is partitioned into two disjoint sets $\CP_1 \sqcup \CP_2$ (which may be empty), we will write:
\begin{equation}
\label{eqn:partition}
\int_{\CP_1 \prec z \prec \CP_2} F(z) = \sum_{c \in \CP_1} \underset{z=c}{\text{Res}} \ \frac {F(z)}z = - \sum_{c \in \CP_2} \underset{z=c}{\text{Res}} \ \frac {F(z)}z
\end{equation}
The first equality is a definition, and the second equality is the residue theorem. If $F(z_1,\dots,z_n)$ is a rational function with simple poles of the form $z_i = c$ and $z_i = \gamma z_j$ for various $c \in \CP_1 \sqcup \CP_2$ and various scalars $\gamma$ in some set $\CQ$, then we set:
\begin{equation}
\label{eqn:partition many}
\int_{\CP_1 \prec z_1 \prec \dots \prec z_n \prec \CP_2} F(z_1,\dots,z_n) 
\end{equation}
as the result of the $n$-step process which starts with $\frac {F(z_1,\dots,z_n)}{z_1\dots z_n}$, and at the $i$-th step replaces a rational function in $z_i,\dots,z_n$ by the sum of its residues of the form $z_i = c \gamma_1\dots \gamma_{i-1}$ for various $c \in \CP_1$ and $\gamma_1,\dots,\gamma_{i-1} \in \CQ \cup \{1\}$. Just like in \eqref{eqn:partition}, the residue theorem implies that the answer is the same as $(-1)^n$ times the result of the $n$-step process which starts with $\frac {F(z_1,\dots,z_n)}{z_1\dots z_n}$, and at the $i$-th step replaces a rational function in $z_1,\dots,z_{n+1-i}$ by the sum of its residues of the form $z_{n+1-i} = c \gamma_1 \dots \gamma_{i-1}$ for various $c \in \CP_2$ and $\gamma_1,\dots,\gamma_{i-1} \in \CQ \cup \{1\}$. \\

\begin{proposition}
\label{prop:corr}

(\cite[following the proof of Proposition 5.12]{W surf}) We have the following formulas for the maps \eqref{eqn:zn projections}:
\begin{multline}
(p_{+} \times p_S)_* r(\CL_1,\dots,\CL_n) = \label{eqn:push one} \\
= \int_{\{0, \infty\} \sqcup \CP \prec z_n \prec \dots \prec z_1 \prec \CU} \frac {r(z_1,\dots,z_n) \prod_{i=1}^n \wedge^\bullet \left(\frac {z_iq}{\CU} \right)}{\left(1 - \frac {z_2 q}{z_1} \right) \dots \left(1 - \frac {z_n q}{z_{n-1}} \right) \prod_{1 \leq i < j \leq n} \zeta \left(\frac {z_j}{z_i} \right)}
\end{multline}
\begin{multline}
(p_{-} \times p_S)_* r(\CL_1,\dots,\CL_n) = \label{eqn:push two} \\
= \int_{\CU \prec z_n \prec \dots \prec z_1 \prec \{0, \infty\} \sqcup \CP} \frac {r(z_1,\dots,z_n) \prod_{i=1}^n \wedge^\bullet \left(- \frac {\CU}{z_i} \right)}{\left(1 - \frac {z_2 q}{z_1} \right) \dots \left(1 - \frac {z_n q}{z_{n-1}} \right) \prod_{1 \leq i < j \leq n} \zeta \left(\frac {z_j}{z_i} \right)} 
\end{multline}
where:
$$
\zeta(x) = \frac {(1-x q_1)(1-x q_2)}{(1-x)(1-x q)} \in \ks(x)
$$
and $r(z_1,\dots,z_n)$ is a rational function with coefficients in $(p_\pm \times p_S)^*(\kms)$, whose poles are all of the form $z_i = c$, where $c \in \{0,\infty\} \sqcup \CP$ for some finite set $\CP$. \\

\end{proposition}

\noindent Note that the integrands in \eqref{eqn:push one}--\eqref{eqn:push two} have poles when $z_i$ equals $q^{1\text{ or }0}$ times one of the Chern roots of $\CU$. Thus, the location of the symbol ``$\CU$" in the subscripts of the integrals \eqref{eqn:push one}--\eqref{eqn:push two} indicates whether these poles are thought to lie in the set $\CP_1$ or $\CP_2$ for the sake of the notation \eqref{eqn:partition}. \\


\section{Computing the Ext operator}
\label{sec:ext}

\medskip

\subsection{} \label{sub:exterior} To properly define the Ext operator \eqref{eqn:a correspondence}, note that the complex $\CE_m$ of \eqref{eqn:e} can be written as a difference $\CV_1 - \CV_2$ of vector bundles. Then we define:
\begin{equation}
\label{eqn:second ext}
\wedge^\bullet \left( \frac {\CE_m}t \right) = \frac {\wedge^\bullet \left( \frac {\CV_1}t \right)}{\wedge^\bullet \left( \frac {\CV_2}t \right)} = \frac {\sum_{k=0}^{\text{rank }\CV_1} (-t)^{-k} [\wedge^k\CV_1]}{\sum_{k=0}^{\text{rank }\CV_2} (-t)^{-k}[\wedge^k\CV_2]}
\end{equation}
and interpret it as a rational function in $t$, with coefficients in $K_{\CM \times \CM'}$. Strictly speaking, the object $\wedge^\bullet \CE_m$ in \eqref{eqn:a correspondence} refers to the specialization of this rational function at $t = 1$. If this specialization is not well-defined (i.e. if $\sum_{k=0}^{\text{rank }\CV_2} (-1)^k [\wedge^k\CV_2]$ is not a unit in $K_{\CM \times \CM'}$), then we employ the following artifice: replace $m$ by $\frac mt$ in formulas \eqref{eqn:def gamma}, \eqref{eqn:comm a}, \eqref{eqn:comm phi} and throughout the current Section. Once one does this, then our main Theorems \ref{thm:main}, \ref{thm:main heis} and Corollary \ref{cor:main} will be equalities of operator-valued rational functions in $t$. Moreover, we will often use the notation:
$$
\wedge^\bullet \left( \frac t{\CU} \right) \qquad \text{instead of} \qquad \wedge^\bullet \left(\CU^\vee t \right)
$$
for any coherent sheaf $\CU$ (all our coherent sheaves have finite projective dimension). \\

\subsection{} The main goal of the present Section is to compute the commutation relations between the Ext operator $A_m : \kmm \longrightarrow \km$ of \eqref{eqn:a correspondence} and the operators:
\begin{equation}
\label{eqn:generators}
W_{n,k}, P_{\pm n'} : \km \longrightarrow \kms
\end{equation}
of \eqref{eqn:op w}, \eqref{eqn:op p 1}, \eqref{eqn:op p 2} for all $n \in \BZ$ and $n',k \in \BN$. One must be careful what one means by ``commutation relation". While the operator:
\begin{align*}
&P_{\pm n} A_m & &\text{unambiguously refers to} & &\kmm \xrightarrow{A_m} \km \xrightarrow{P_{\pm n}} \kms  \\
& A_m P_{\pm n} & &\text{henceforth refers to} & &\kmm \xrightarrow{P_{\pm n}} \kmms \xrightarrow{A_m \times \text{Id}_S} \kms 
\end{align*}
and analogously for $W_{n,k}$ instead of $P_{\pm n}$. As opposed from the operators \eqref{eqn:generators}, the operator $A_m$ acts non-trivially between all components of the moduli space:
\begin{equation}
\label{eqn:a in components}
A_m |_c^{c'} : K_{\CM_{c'}} \longrightarrow K_{\CM_c}
\end{equation}
In principle, the moduli spaces of sheaves in the domain and codomain can correspond to different choices of first Chern class and stability condition, but we always require them to have the same rank $r$. Therefore, there are two universal sheaves:
$$
\xymatrix{\CU \ar@{..>}[d] \\ \CM \times S} \qquad \qquad \xymatrix{\CU' \ar@{..>}[d] \\ \CM' \times S} 
$$
of the same rank $r$, where $\CM$ (respectively $\CM'$) is the union of the moduli spaces that appear in the codomain (respectively domain) of \eqref{eqn:a in components}. The determinants of these universal sheaves are denoted by $u$ and $u'$, respectively, as in \eqref{eqn:def u}. \\

\subsection{} 

We must explain how to make sense of the symbols $q,m,\gamma$ in \eqref{eqn:comm a}, \eqref{eqn:comm a heis 1}, \eqref{eqn:comm a heis 2}. In the language of correspondences from Subsection \ref{sub:correspondences}, the operators:
$$
\kmm \xrightarrow{z} \kms
$$
studied in the present paper (such as the compositions $W_{n,k} A_m$ or $P_{\pm n} A_m$ that appear in \eqref{eqn:comm a}, \eqref{eqn:comm a heis 1}, \eqref{eqn:comm a heis 2}) arise from $K$--theory classes $\Gamma$ on $\CM \times \CM' \times S$. Then the product $q z$ refers to the operator corresponding to the class $\text{proj}_S^*(q) \cdot \Gamma$, while the product $\gamma z$ refers to the operator corresponding to the class: 
$$
\text{proj}_{S}^*\left(\frac mq\right)^r \cdot \frac {\proj_{\CM \times S}^*(\det \CU)}{\proj_{\CM' \times S}^*(\det \CU')} \cdot \Gamma
$$
where $\CM\times \CM' \times S \xrightarrow{\proj_{\CM \times S}, \proj_{\CM' \times S}, \proj_S} \CM \times S, \CM' \times S, S$ are the projections. \\

\begin{proposition}
\label{prop:doesn't matter}

We have the equality of correspondences $K_{\CM_{c \pm n}} \rightarrow K_{\CM_c \times S}$:
\begin{equation}
\label{eqn:doesn't matter}
P_{\pm n} \cdot \left( \det \CU_{c \pm n} \right) = \left( \det \CU_c \right) \cdot P_{\pm n}
\end{equation}
for all $c \in \BZ$. Formula \eqref{eqn:doesn't matter} also holds with $P_{\pm n}$ replaced by $W_{n,k}$ or $H_{\pm n}$. \\
\end{proposition}

\noindent Equation \eqref{eqn:doesn't matter} is best restated in the language of correspondences from Subsection \ref{sub:correspondences}. In these terms, $P_{\pm n}$ is given by a $K$--theory class supported on the locus: 
$$
\mathfrak{C} = \Big\{ ( \CF_{c + n} \subset_{nx} \CF_c ) \text{ for some }x \in S \Big\} \subset \CM_{c + n} \times \CM_c \times S
$$
where $\CF' \subset_{nx} \CF$ means that $\CF' \subset \CF$ and that $\CF/\CF'$ is a length $n$ sheaf supported at $x$. Then \eqref{eqn:doesn't matter} merely states that the universal sheaves $\CU_{c + n}$ and $\CU_c$ have isomorphic determinants when restricted to $\mathfrak{C}$. This is just the version ``in families" of the well-known statement that a codimension 2 modification of a torsion-free sheaf does not change its determinant. As a consequence of Proposition \ref{prop:doesn't matter}, $\gamma$ of \eqref{eqn:def gamma} will behave just like a constant in all our computations, i.e. it will not matter where we insert $\gamma$ in any product of operators among $P_{\pm n}, H_{\pm n}, W_{n,k}$. \\

\subsection{} Our main intersection-theoretic computation is the following: \\

\begin{lemma}
\label{lem:commute heis}

We have the following relations involving the \emph{Ext} operator $A_m$:
\begin{equation}
\label{eqn:one heis}
A_m (H_{-n} - H_{-n+1}) = \gamma^n (H_{-n} - H_{-n+1}) A_m 
\end{equation}
\begin{equation}
\label{eqn:two heis}
\qquad \qquad A_m \left(H_n - H_{n-1} \gamma^{-1} \right) = \left( H_n \gamma^{-n} - H_{n-1} q^r \gamma^{-n+1} \right) A_m
\end{equation}
for all $n \in \BN$ (recall that $H_0 = \emph{proj}_1^*$, where $\CM \times S \xrightarrow{\emph{proj}_1} \CM$ is the usual projection). \\

\end{lemma}

\begin{proof}

Consider the following diagrams of spaces and arrows, for all $c, c' \in \BZ$:

\begin{equation}
\label{eqn:big diagram 1}
\xymatrix{ & & \CM_c \times S \times \CM_{c'} \ar@/_3pc/[llddd]_{\pi_1 \times \Id_S} \ar@/^3pc/[rrddd]^{\pi_2} & & \\
& & \CM_c \times \fZ^\bullet_{c'+n,c'} \ar[ld]_{\Id \times p_+ \times p_S} \ar[rd]   \ar[u]^{\Id \times p_S \times p_-} & & \\
& \CM_c \times \CM_{c'+n} \times S \ar[ld] \ar[rd] & & \fZ^\bullet_{c'+n,c'} \ar[ld]^{p_+ \times p_S} \ar[rd]_{p_-} & \\ 
\CM_c \times S & & \CM_{c'+n} \times S & & \CM_{c'}}
\end{equation}
\begin{equation}
\label{eqn:big diagram 2}
\xymatrix{ & & \CM_c \times S \times \CM_{c'}  \ar@/_3pc/[llddd]_{\pi_1' \times \text{Id}_S} \ar@/^3pc/[rrddd]^{\pi'_2} & & \\
& & \fZ^\bullet_{c,c-n} \times \CM_{c'} \ar[ld] \ar[rd]^{p_-' \times \Id}   \ar[u]_{p_+' \times p_S' \times \Id} & & \\
& \fZ^\bullet_{c,c-n}\ar[ld]^{p'_+ \times p_S'} \ar[rd]_{p'_-} & & \CM_{c-n} \times \CM_{c'} \ar[ld] \ar[rd] & \\ 
\CM_c \times S & & \CM_{c-n} & & \CM_{c'}}
\end{equation}
Recall that $H_{-n} = (p_+ \times p_S)_* p_-^*$, in the notation of \eqref{eqn:zn projections}. Then the rule for composition of correspondences in \eqref{eqn:composition corr} gives us the following formulas:
\begin{equation}
\label{eqn:comp 1}
A_m H_{-n} = (\pi_1 \times \Id_S)_*(\Upsilon_n \cdot \pi_2^{*})
\end{equation}
\begin{equation}
\label{eqn:comp 2}
H_{-n} A_m = (\pi'_1 \times \Id_S)_*(\Upsilon'_n \cdot \pi_2'^{*})
\end{equation}
where, in the notation of \eqref{eqn:big diagram 1} and \eqref{eqn:big diagram 2}:
\begin{align}
&\Upsilon_n = \left( \Id \times p_S \times p_- \right)_* \Big[ \wedge^\bullet \left( (\Id \times p_+)^* \CE_{m} \right) \Big] \label{eqn:gamma 1} \\
&\Upsilon'_n = \left( p_+' \times p_S' \times \Id \right)_* \Big[  \wedge^\bullet \left( (p'_- \times \Id)^* \CE_{m} \right) \Big] \label{eqn:gamma 2}
\end{align}
are certain classes on $\CM_c \times S \times \CM_{c'}$, that we will now compute. \\

\begin{claim}
\label{claim:joe}

We have the following equalities in $K$--theory:
\begin{equation}
\label{eqn:geo 1}
(\emph{Id} \times p_+)^* \CE_{m} = (\emph{Id} \times p_-)^* \CE_{m} + \left(\frac 1{\CL_1} +\dots+ \frac 1{\CL_n} \right) (\emph{Id} \times p_S)^*\left( \frac {\CU m}q \right) 
\end{equation}
on $\CM_c \times \fZ_{c'+n,c'}^\bullet$ (where $\CU$ denotes the universal sheaf on $\CM_c \times S$), and:
\begin{equation}
\label{eqn:geo 2}
(p_-' \times \emph{Id})^* \CE_{m} = (p_+' \times \emph{Id})^* \CE_{m} - (\CL_1+\dots+\CL_n) (p'_S \times \emph{Id})^*\left( {\CU'}^\vee m \right) 
\end{equation}
on $\fZ_{c,c-n}^\bullet \times \CM_{c'}$ (where $\CU'$ denotes the universal sheaf on $\CM_{c'} \times S$). \\

\end{claim}

\begin{proof}

To prove \eqref{eqn:geo 1}, consider the following diagram:
\begin{equation}
\label{eqn:grothendieck 1}
\xymatrix{ & \CM_c \times \fZ_{c'+n,c'}^\bullet \times S \ar[dd]^{\rho} \ar[ld]_-{\text{Id} \times p_+ \times \text{Id}_S} \ar[rd]^-{\text{Id} \times p_- \times \text{Id}_S} & \\ 
\CM_c \times \CM_{c'+n} \times S \ar[dd]^{\rho} & & \CM_c \times \CM_{c'} \times S \ar[dd]^{\rho} \\
 & \CM_c \times \fZ_{c'+n,c'}^\bullet  \ar[ld]_-{\text{Id} \times p_+} \ar[rd]^-{\text{Id} \times p_-} & \\ 
\CM_c \times \CM_{c'+n} & & \CM_c \times \CM_{c'}}
\end{equation}
where the vertical maps are the natural projections (we use the notation $\rho$ for all of them). We have the following short exact sequence of sheaves over $\fZ_{c'+n,c'}^{\bullet} \times S$:
\begin{equation}
\label{eqn:joe}
0 \longrightarrow \CU'_+ \longrightarrow \CU'_- \longrightarrow \Gamma_*(\CL_1 ``\oplus" \dots ``\oplus" \CL_n) \longrightarrow 0
\end{equation}
where $\CU'_\pm = (p_\pm^* \times \text{Id}_S)(\text{universal sheaf})$, while $\CL_1,\dots,\CL_n$ denote the tautological line bundles on $\fZ_{c'+n,c'}^\bullet$ that were defined in \eqref{eqn:line bundles}, and: 
\begin{equation}
\label{eqn:def Gamma}
\Gamma : \fZ_{c'+n,c'}^\bullet \rightarrow \fZ_{c'+n,c'}^\bullet \times S
\end{equation}
is the graph of the map $p_S$. The notation $``\oplus"$ in \eqref{eqn:joe} refers to a coherent sheaf which is filtered by the line bundles $\CL_1$,\dots,$\CL_n$; since we work in $K$--theory, we henceforth make no distinction between this coherent sheaf and its associated graded object. We may also pull-back the short exact sequence \eqref{eqn:joe} to $\CM_c \times \fZ_{c'+n,c'}^\bullet \times S$. Now apply the functor $\sRHom(-,\CU \otimes m)$ to the short exact sequence \eqref{eqn:joe}, where $\CU$ is the universal sheaf pulled back from $\CM_c \times S$:
$$
\sRHom(\CU'_+,\CU \otimes m) = \sRHom(\CU'_-,\CU \otimes m) - \sum_{i=1}^n \frac 1{\CL_i} \sRHom(\CO_\Gamma,\CU \otimes m)
$$
Now recall that the line bundles $\CL_i$ come from $\fZ_{c'+n,c'}^\bullet$, and so they are unaffected by the derived push-forward map $\rho_*$:
$$
\rho_* \sRHom(\CU'_+,\CU \otimes m) = \rho_* \sRHom(\CU'_-,\CU \otimes m) - \sum_{i=1}^n \frac 1{\CL_i} \rho_* \sRHom(\CO_\Gamma,\CU \otimes m)  
$$
Recalling \eqref{eqn:sheaf e}, the formula above reads:
\begin{equation}
\label{eqn:kay}
(\text{Id} \times p_+)^* \CE_{m} = (\text{Id} \times p_-)^* \CE_{m} + \sum_{i=1}^n \frac 1{\CL_i} \rho_* \sRHom(\CO_\Gamma,\CU \otimes m)
\end{equation}
Then \eqref{eqn:geo 1} follows from the fact that:
\begin{equation}
\label{eqn:rick}
\rho_* \sRHom(\CO_\Gamma,\CU \otimes m) = \underbrace{\rho_* \circ \Gamma_*}_{\text{Id}} \left( \sRHom(\CO, \Gamma^! (\CU \otimes m)) \right) = \CU m \Big|_\Gamma \otimes \Gamma^! \CO 
\end{equation}
(the first equality is coherent duality, and the second equality holds for any closed embedding $\Gamma$). The right-hand side of \eqref{eqn:rick} matches $(\text{Id} \times p_S)^*(\CU m/q)$ because the map $\Gamma: \fZ_n^\bullet \rightarrow \fZ_n^\bullet \times S$ is obtained by base change from the diagonal map $S \rightarrow S \times S$, and the ratio of dualizing objects on $S$ and $S \times S$ is precisely $q = [\omega_S]$. \\

\noindent As for \eqref{eqn:geo 2}, consider the diagram:
\begin{equation}
\label{eqn:grothendieck 2}
\xymatrix{ & \fZ_{c,c-n}^\bullet \times \CM_{c'} \times S \ar[dd]^{\rho} \ar[ld]_-{p'_+ \times \text{Id} \times \text{Id}_S} \ar[rd]^-{ p'_- \times \text{Id} \times \text{Id}_S} & \\ 
\CM_c \times \CM_{c'} \times S \ar[dd]^{\rho} & & \CM_{c-n} \times \CM_{c'} \times S \ar[dd]^{\rho} \\
 & \fZ_{c,c-n}^\bullet \times \CM_{c'}  \ar[ld]_-{p'_+ \times \text{Id}} \ar[rd]^-{p'_- \times \text{Id}} & \\ 
\CM_c \times \CM_{c'} & & \CM_{c-n} \times \CM_{c'}}
\end{equation}
and consider the following analogue of \eqref{eqn:joe}:
$$
0 \longrightarrow \CU_+ \longrightarrow \CU_- \longrightarrow \Gamma'_*(\CL_1 ``\oplus" \dots ``\oplus" \CL_n) \longrightarrow 0
$$
where $\CU_\pm = ({p'}_\pm^* \times \text{Id}_S)(\CU)$, and $\Gamma'$ denotes the graph of the map $p_S : \fZ_{c,c-n}^\bullet \rightarrow S$. Let us apply the functor $\sRHom(\CU', - \otimes m)$ to the short exact sequence above:
$$
\sRHom(\CU',\CU_- \otimes m) = \sRHom(\CU',\CU_+ \otimes m) + \sum_{i=1}^n \CL_i \otimes \sRHom(\CU',\CO_{\Gamma'} \otimes m)
$$
Let us apply $\rho_*$ to the equality above, and recall the definition of $\CE_m$ in \eqref{eqn:sheaf e}:
$$
(p_-' \times \text{Id})^*\CE_m = (p_+' \times \text{Id})^*\CE_m - \sum_{i=1}^n \CL_i \otimes \rho_* \sRHom(\CU',\CO_{\Gamma'} \otimes m)
$$
By adjunction, we have:
$$
\rho_* \sRHom(\CU',\CO_{\Gamma'} \otimes m) = \underbrace{\rho_* \circ \Gamma'_*}_{\text{Id}} \sRHom(\CU'|_{\Gamma'}, {p'_S}^*m) =  \left({\CU'}^\vee m \right) \Big|_{\Gamma'}
$$

\end{proof}

\noindent Armed with \eqref{eqn:geo 1} and \eqref{eqn:geo 2}, we may rewrite \eqref{eqn:gamma 1} and \eqref{eqn:gamma 2} as:
\begin{align*}
&\Upsilon_n = \left[ \wedge^\bullet\CE_m \right] \cdot \left( \Id \times p_S \times p_- \right)_* \left[  \bigotimes_{i=1}^n \wedge^\bullet \left( \frac {\CU m}{\CL_i q}\right) \right] \\
&\Upsilon_n' = \left[ \wedge^\bullet\CE_m \right] \cdot \left( p_+' \times p_S' \times \Id \right)_* \left[ \bigotimes_{i=1}^n \wedge^\bullet \left(- \frac {\CL_i m}{\CU'}\right) \right]
\end{align*}
Therefore, Proposition \ref{prop:corr} implies (henceforth, ``$\CU, \CU'$" in the subscript of the integrals are simply shorthand for ``the set of Chern roots of $\CU, \CU'$", respectively):
\begin{align}
&\Upsilon_n = \left[ \wedge^\bullet\CE_m \right] \int_{\CU' \prec z_n \prec \dots \prec z_1 \prec \{0, \infty\} \sqcup \CU} \frac  {\prod_{i=1}^n \frac {\wedge^\bullet \left(\frac {\CU m}{z_i q} \right)}{\wedge^\bullet \left(\frac {\CU'}{z_i} \right)}}{\prod_{i=1}^{n-1} \left(1 - \frac {qz_{i+1}}{z_i} \right) \prod_{i < j} \zeta \left( \frac {z_j}{z_i} \right)}  \label{eqn:gam 1} \\
&\Upsilon_n' = \left[ \wedge^\bullet\CE_m \right] \int_{\{0,\infty\} \sqcup \CU' \prec z_n \prec \dots \prec z_1 \prec \CU} \frac {\prod_{i=1}^n \frac {\wedge^\bullet \left(\frac {z_i q}{\CU} \right)}{\wedge^\bullet \left(\frac {z_i m}{\CU'} \right)}}{\prod_{i=1}^{n-1} \left(1 - \frac {qz_{i+1}}{z_i} \right) \prod_{i < j} \zeta \left( \frac {z_j}{z_i} \right)}  \label{eqn:gam 2}
\end{align}
Consider the following rational function with coefficients in $K_{\CM_c \times S \times \CM_{c'}}$:
\begin{equation}
\label{eqn:definition i}
I_n(z_1,\dots,z_n) = \frac {\prod_{i=1}^n \frac {\wedge^\bullet \left(\frac {\CU m}{z_i q} \right)}{\wedge^\bullet \left(\frac {\CU'}{z_i} \right)}}{\prod_{i=1}^{n-1} \left(1 - \frac {qz_{i+1}}{z_i} \right) \prod_{i < j} \zeta \left( \frac {z_j}{z_i} \right)} 
\end{equation}
One may then rewrite \eqref{eqn:gam 1} and \eqref{eqn:gam 2} as:
\begin{align*}
&\Upsilon_n = \left[ \wedge^\bullet\CE_m \right] \int_{\CU' \prec z_n \prec \dots \prec z_1 \prec \{0, \infty\} \sqcup \CU} I_n(z_1,\dots,z_n) \\
&\Upsilon_n' = \left[ \wedge^\bullet\CE_m \right] \int_{\{0,\infty\} \sqcup \CU' \prec z_n \prec \dots \prec z_1 \prec \CU} I_n(z_1m,\dots,z_nm)\cdot \gamma^{-n} 
\end{align*}
Changing the variables $z_i \mapsto \frac {z_i}m$ in the second formula, we conclude that:
\begin{equation}
\label{eqn:integral identity 1}
\Upsilon_n - \Upsilon_n' \cdot \gamma^n =
\end{equation}
$$
= \left[ \wedge^\bullet \CE_m \right] \left[ \int_{\CU' \prec z_n \prec \dots \prec z_1 \prec \{0, \infty\} \sqcup \CU} I_n  - \int_{\{0,\infty\} \sqcup \CU' \prec z_n \prec \dots \prec z_1 \prec \CU} I_n \right]
$$
The only difference between the two integrals is the location of the poles $\{0,\infty\}$ with respect to the variables $z_1,\dots,z_n$. Therefore, we conclude that the difference above picks up the residues at 0 and $\infty$ in the various variables. However, all such residues vanish, except for:
\begin{align}
&\underset{z_1 = \infty}{\text{Res}} \frac {I_n(z_1,\dots,z_n)}{z_1} = - I_{n-1}(z_2,\dots,z_n) \label{eqn:residue i 1} \\
&\underset{z_n = 0}{\text{Res}} \frac {I_n(z_1,\dots,z_n)}{z_n} = \gamma \cdot I_{n-1}(z_1,\dots,z_{n-1}) \label{eqn:residue i 2}
\end{align}
Therefore, formula \eqref{eqn:integral identity 1} implies that:
\begin{equation}
\label{eqn:nice}
\Upsilon_n - \Upsilon_n' \cdot \gamma^n  = \Upsilon_{n-1} - \Upsilon_{n-1}' \cdot \gamma^n
\end{equation}
which, as an equality of classes on $\CM_c \times S \times \CM_{c'}$, precisely encodes \eqref{eqn:one heis}. Let us run the analogous computation for \eqref{eqn:two heis} (we will recycle all of our notations):


\begin{equation}
\label{eqn:big diagram 1 bis}
\xymatrix{ & & \CM_c \times S \times \CM_{c'} \ar@/_3pc/[llddd]_{\pi_1 \times \Id_S} \ar@/^3pc/[rrddd]^{\pi_2} & & \\
& & \CM_c \times \fZ^\bullet_{c',c'-n} \ar[ld]_{\Id \times p_- \times p_S} \ar[rd]   \ar[u]^{\Id \times p_S \times p_+} & & \\
& \CM_c \times \CM_{c'-n} \times S \ar[ld] \ar[rd] & & \fZ^\bullet_{c',c'-n} \ar[ld]^{p_- \times p_S} \ar[rd]_{p_+} & \\ 
\CM_c \times S & & \CM_{c'-n} \times S & & \CM_{c'}}
\end{equation}
\begin{equation}
\label{eqn:big diagram 2 bis}
\xymatrix{ & & \CM_c \times S \times \CM_{c'}  \ar@/_3pc/[llddd]_{\pi_1' \times \text{Id}_S} \ar@/^3pc/[rrddd]^{\pi'_2} & & \\
& & \fZ^\bullet_{c+n,c} \times \CM_{c'} \ar[ld] \ar[rd]^{p_+' \times \Id}   \ar[u]_{p_-' \times p_S' \times \Id} & & \\
& \fZ^\bullet_{c+n,c}\ar[ld]^{p'_- \times p_S'} \ar[rd]_{p'_+} & & \CM_{c+n} \times \CM_{c'} \ar[ld] \ar[rd] & \\ 
\CM_c \times S & & \CM_{c+n} & & \CM_{c'}}
\end{equation}
Recall that $H_{n} = (-1)^{rn} u^n (p_- \times p_S)_* (\CQ^{-r} \cdot p_+^*)$, in the notation of \eqref{eqn:zn projections}. Then the rule for composition of correspondences in \eqref{eqn:composition corr} gives us the following:
\begin{equation}
\label{eqn:comp 1 bis}
A_m H_{n} = (\pi_1 \times \Id_S)_*(\Upsilon_n \cdot \pi_2^{*})
\end{equation}
\begin{equation}
\label{eqn:comp 2 bis}
H_{n} A_m = (\pi'_1 \times \Id_S)_*(\Upsilon'_n \cdot \pi_2'^{*})
\end{equation}
where:
\begin{align}
&\Upsilon_n = (-1)^{rn} {u'}^{n} \left(\Id \times p_S \times p_+ \right)_* \Big[\CQ^{-r} \cdot  \wedge^\bullet \left( (\Id \times p_-)^* \CE_{m} \right) \Big] \label{eqn:gamma 1 bis} \\
&\Upsilon'_n = (-1)^{rn} u^{n}  \left(p_-' \times p_S' \times \Id \right)_* \Big[\CQ^{-r} \cdot \wedge^\bullet \left( (p'_+ \times \Id)^* \CE_{m} \right) \Big] \label{eqn:gamma 2 bis}
\end{align}
are certain classes on $\CM_c \times S \times \CM_{c'}$.  As a consequence of \eqref{eqn:geo 1} and \eqref{eqn:geo 2}, which continue to hold as stated in the new setup, we may rewrite \eqref{eqn:gamma 1 bis} and \eqref{eqn:gamma 2 bis} as:
\begin{align*}
&\Upsilon_n = (-1)^{rn} {u'}^{n} \left[ \wedge^\bullet\CE_m \right] \left( \Id \times p_S \times p_+ \right)_* \left[\CQ^{-r} \bigotimes_{i=1}^n \wedge^\bullet \left(-\frac {\CU m}{\CL_i q}\right) \right] \\
&\Upsilon_n' = (-1)^{rn} u^{n} \left[ \wedge^\bullet\CE_m \right] \left( p_-' \times p_S' \times \Id \right)_* \left[\CQ^{-r} \bigotimes_{i=1}^n \wedge^\bullet \left(\frac {\CL_i m}{\CU'}\right) \right]
\end{align*}
Therefore, Proposition \ref{prop:corr} implies:
\begin{align}
&\Upsilon_n = \left[ \wedge^\bullet\CE_m \right] \int_{\{0,\infty\} \sqcup \CU \prec z_n \prec \dots \prec z_1 \prec \CU'} \frac  {(-1)^{rn} {u'}^{n} z_1^{-r} \dots z_n^{-r} \prod_{i=1}^n \frac {\wedge^\bullet \left(\frac {z_iq}{\CU'} \right)}{\wedge^\bullet \left(\frac {\CU m}{z_i q} \right)}}{\prod_{i=1}^{n-1} \left(1 - \frac {qz_{i+1}}{z_i} \right) \prod_{i < j} \zeta \left( \frac {z_j}{z_i} \right)}  \label{eqn:gam 1 bis} \\
&\Upsilon_n' = \left[ \wedge^\bullet\CE_m \right] \int_{\CU \prec z_n \prec \dots \prec z_1 \prec \{0,\infty\} \sqcup \CU'} \frac {(-1)^{rn} u^{n} z_1^{-r} \dots z_n^{-r} \prod_{i=1}^n \frac {\wedge^\bullet \left(\frac {z_i m}{\CU'} \right)}{\wedge^\bullet \left(\frac {\CU}{z_i} \right)}}{\prod_{i=1}^{n-1} \left(1 - \frac {qz_{i+1}}{z_i} \right) \prod_{i < j} \zeta \left( \frac {z_j}{z_i} \right)}  \label{eqn:gam 2 bis}
\end{align}
Consider the following rational function with coefficients in $K_{\CM_c \times S \times \CM_{c'}}$:
\begin{equation}
\label{eqn:definition i bis}
I_n(z_1,\dots,z_n) = \frac {q^{rn} \prod_{i=1}^n \frac {\wedge^\bullet \left(\frac {\CU'}{z_iq} \right)}{\wedge^\bullet \left(\frac {\CU m}{z_i q} \right)}}{\prod_{i=1}^{n-1} \left(1 - \frac {qz_{i+1}}{z_i} \right) \prod_{i < j} \zeta \left( \frac {z_j}{z_i} \right)} 
\end{equation}
One may then rewrite \eqref{eqn:gam 1 bis} and \eqref{eqn:gam 2 bis} as:
\begin{align*}
&\Upsilon_n = \left[ \wedge^\bullet\CE_m \right] \int_{\{0,\infty\} \sqcup \CU \prec z_n \prec \dots \prec z_1 \prec \CU'} I_n(z_1,\dots,z_n) \\
&\Upsilon_n' = \left[ \wedge^\bullet\CE_m \right] \int_{\CU \prec z_n \prec \dots \prec z_1 \prec \{0,\infty\} \sqcup \CU'} I_n\left( \frac {z_1m}q ,\dots, \frac {z_nm}q \right) \cdot \gamma^n
\end{align*}
Changing the variables $z_i \mapsto \frac {z_i q}m$ in the second formula, we conclude that:
\begin{equation}
\label{eqn:integral identity 1 bis}
\Upsilon_n - \Upsilon_n' \cdot \gamma^{-n} =
\end{equation}
$$
= \left[ \wedge^\bullet \CE_m \right] \left[ \int_{\{0,\infty\} \sqcup \CU \prec z_n \prec \dots \prec z_1 \prec \CU'} I_n  - \int_{\CU \prec z_n \prec \dots \prec z_1 \prec \{0,\infty\} \sqcup \CU'} I_n \right]
$$
The only difference between the two integrals is the location of the poles $\{0,\infty\}$ with respect to the variables $z_1,\dots,z_n$. Therefore, we conclude that the difference above picks up the residues at 0 and $\infty$ in the various variables. However, all such residues vanish, except for:
\begin{align*}
&\underset{z_n = 0}{\text{Res}} \frac {I_n(z_1,\dots,z_n)}{z_n} = \gamma^{-1} \cdot I_{n-1}(z_1,\dots,z_{n-1}) \\
&\underset{z_1 = \infty}{\text{Res}} \frac {I_n(z_1,\dots,z_n)}{z_1} = - q^r \cdot I_{n-1}(z_2,\dots,z_n) 
\end{align*}
Therefore, formula \eqref{eqn:integral identity 1 bis} implies that:
\begin{equation}
\label{eqn:nice bis}
\Upsilon_n - \Upsilon_n' \cdot \gamma^{-n}  = \Upsilon_{n-1} \cdot \gamma^{-1} - \Upsilon'_{n-1} \cdot q^r \gamma^{-n+1}
\end{equation}
which, as an equality of classes on $\CM_c \times S \times \CM_{c'}$, precisely encodes \eqref{eqn:two heis}. \\

\end{proof}

\subsection{}

We will now show how Lemma \ref{lem:commute heis} allows us to prove Theorem \ref{thm:main heis}. \\

\begin{proof} \emph{of Theorem \ref{thm:main heis}:} We will only prove \eqref{eqn:comm a heis 1}, since \eqref{eqn:comm a heis 2} is analogous. We will use formulas \eqref{eqn:h to p}, which say that the $H$ operators are to the $P$ operators as complete symmetric functions are to power sum functions. Then let us place \eqref{eqn:one heis} into a generating series that goes over all $n \in \BN$:
\begin{equation}
\label{eqn:kc}
\sum_{n=0}^\infty A_m H_{-n} \left( z^n - z^{n+1} \right) = \sum_{n=0}^\infty \left( \left( \gamma z \right)^n - \left( \gamma z \right)^{n+1} \right)  H_{-n} A_m 
\end{equation}
If we write $H_-(z)$ for the power series \eqref{eqn:h to p} (with sign $\pm = -$), then \eqref{eqn:kc} reads:
\begin{equation}
\label{eqn:a and h series}
A_m H_-(z) (1 - z) = H_-\left(z \gamma\right) (1-\gamma z) A_m
\end{equation}
If $P$ is an operator $\km \rightarrow \kms$ which commutes with two line bundles $\ell$ and $\ell'$ (in the sense of Proposition \ref{prop:doesn't matter}, and the discussion after it), then:
\begin{equation}
\label{eqn:equivalence}
A \exp(P) \exp(\ell') \Big|_\Delta = \exp(P) \exp(\ell) \Big|_\Delta A \qquad \Leftrightarrow \qquad AP + A \ell' = P A + \ell A
\end{equation}
(this claim uses the associativity of the operation $x,y \leadsto xy|_\Delta$, as discussed in Subsection \ref{sub:action}). With this in mind, formula \eqref{eqn:a and h series} implies:
$$
A_m P_-(z) - \sum_{n=1}^\infty \frac {A_m}{n z^{-n}} = P_-\left(z \gamma\right) A_m - \sum_{n=1}^\infty \gamma^n \frac {A_m}{n z^{-n}}
$$
where $P_-(z) = \sum_{n=1}^\infty \frac {P_{-n}}{nz^{-n}}$. Extracting the coefficient of $z^n$ yields precisely \eqref{eqn:comm a heis 1}. 

\end{proof}

\subsection{} Having proved Lemma \ref{lem:commute heis}, we will now perform the analogous computations for the commutator of $A_m$ with the operators of Subsection \ref{sub:w action}: \\

\begin{lemma}
\label{lem:comm}

We have the following relations involving the \emph{Ext} operator $A_m$:
\begin{multline}
A_m L_n(y) - A_m L_{n-1}(y) = \label{eqn:one} \\
= L_n \left( \frac ym \right) A_m \cdot \gamma^n - L_{n-1} \left( \frac {yq}m \right) E \left( \frac {yq}m \right) A_m E \left( y \right)^{-1} \Big|_\Delta \cdot \gamma^{n-1} 
\end{multline}
and:
\begin{multline}
U_n \left(\frac {yq}m \right) A_m \cdot \gamma^{-n} -  U_{n-1} \left(\frac {yq}m \right) A_m \cdot q \gamma^{-n+1} = \label{eqn:two} \\
=  A_m U_n(y) - E\left(\frac {yq}m\right)^{-1}A_m E(yq) U_{n-1}(yq) \Big|_\Delta \cdot q 
\end{multline}
The two sides of equations \eqref{eqn:one} and \eqref{eqn:two} maps $\kmm$ to $\kms \left\llbracket y^{-1} \right \rrbracket$. The symbol $|_\Delta$ applied to any term that involves three of the series $L,E,U$ means that we restrict a certain operator $\kmm \rightarrow \kmsss \left\llbracket y^{-1} \right \rrbracket$ to the small diagonal. \\

\end{lemma}

\begin{proof} In order to prove \eqref{eqn:one}, we will closely follow the proof of Lemma \ref{lem:commute heis}. With the notation therein, one needs to replace \eqref{eqn:gamma 1} and \eqref{eqn:gamma 2} by:
\begin{align*}
&\Upsilon_{n,y} = \left( \Id \times p_S \times p_- \right)_* \left[\frac 1{1- \frac y{\CL_n}} \wedge^\bullet \left( (\Id \times p_+)^* \CE_{m} \right) \right] \\
&\Upsilon'_{n,y} = \left( p_+' \times p_S' \times \Id \right)_* \left[ \frac 1{1- \frac y{\CL_n}} \wedge^\bullet \left( (p'_- \times \Id)^* \CE_{m} \right) \right]
\end{align*}
This has the effect of inserting:
$$
\left(1 - \frac y{z_n} \right)^{-1}
$$
into the right-hand sides of formulas \eqref{eqn:gam 1} and \eqref{eqn:gam 2}. Therefore, the function $I_n(z_1,\dots,z_n)$ defined in \eqref{eqn:definition i} should be replaced by:
$$
I_{n,y}(z_1,\dots,z_n) = \frac {I_n(z_1,\dots,z_n)}{1- \frac y{z_n}}
$$
It is easy to see that the non-zero residues of $I_{n,y}$ are:
\begin{align*}
&\underset{z_1 = \infty}{\text{Res}} \frac {I_{n,y}(z_1,\dots,z_n)}{z_1} = - I_{n-1,y}(z_2,\dots,z_n) \\
&\underset{z_n = y}{\text{Res}} \frac {I_{n,y}(z_1,\dots,z_n)}{z_n} = \frac {\wedge^\bullet \left(\frac {\CU m}{yq} \right)}{\wedge^\bullet \left( \frac {\CU'}y \right)} \cdot \frac {I_{n-1,yq}(z_1,\dots,z_{n-1})}{\prod_{i=1}^{n-1} \zeta \left( \frac y{z_i} \right)}
\end{align*}
Therefore, the analogue of identity \eqref{eqn:nice} is:
$$
\Upsilon_{n,y} - \Upsilon_{n, \frac ym}' \cdot \gamma^n = \Upsilon_{n-1,y} - \Upsilon_{n-1, \frac {yq}m}' \cdot \gamma^{n-1} \frac {\wedge^\bullet \left(\frac {\CU m}{yq} \right)}{\wedge^\bullet \left( \frac {\CU'}y \right)}
$$
This equality of classes on $\CM_c \times S \times \CM_{c'}$ precisely underlies equality \eqref{eqn:one}. \\

\noindent As for \eqref{eqn:two}, we proceed analogously. One needs to replace \eqref{eqn:gamma 1 bis} and \eqref{eqn:gamma 2 bis} by:
\begin{align*}
&\Upsilon_n = \frac {(-1)^{rn} {u'}^n}{q^{(r-1)n}} \left(\Id \times p_S \times p_+ \right)_* \left[ \frac {\CQ^{-r}}{1- \frac y{\CL_n}} \cdot  \wedge^\bullet \left( (\Id \times p_-)^* \CE_{m} \right) \right] \\
&\Upsilon'_n = \frac {(-1)^{rn} u^n}{q^{(r-1)n}} \left(p_-' \times p_S' \times \Id \right)_* \left[  \frac {\CQ^{-r}}{1- \frac y{\CL_n}} \cdot \wedge^\bullet \left( (p'_+ \times \Id)^* \CE_{m} \right) \right]
\end{align*}
This has the effect of inserting:
$$
q^{n(1-r)}\left(1 - \frac y{z_n} \right)^{-1}
$$
into the right-hand sides of formulas \eqref{eqn:gam 1 bis} and \eqref{eqn:gam 2 bis}. Therefore, the function $I_n$ defined in \eqref{eqn:definition i bis} should be replaced by:
$$
I_{n,y}(z_1,\dots,z_n) = \frac {I_n(z_1,\dots,z_n)}{ q^{(r-1)n} \left(1- \frac y{z_n}\right)}
$$
It is easy to see that the non-zero residues of $I_{n,y}$ are:
\begin{align*}
&\underset{z_n = y}{\text{Res}} \frac {I_{n,y}(z_1,\dots,z_n)}{z_n} = q \frac {\wedge^\bullet \left( \frac {\CU'}{yq} \right)}{\wedge^\bullet \left(\frac {\CU m}{yq} \right)} \cdot \frac {I_{n-1,yq}(z_1,\dots,z_{n-1})}{\prod_{i=1}^{n-1} \zeta \left(\frac y{z_i} \right)} \\
&\underset{z_1 = \infty}{\text{Res}} \frac {I_{n,y}(z_1,\dots,z_n)}{z_1} = - q \cdot I_{n-1,y}(z_2,\dots,z_n) 
\end{align*}
Therefore, the analogue of identity \eqref{eqn:nice bis} is:
$$
\Upsilon_{n,y} - \Upsilon_{n, \frac {yq}m}' \cdot \gamma^{-n} = \Upsilon_{n-1,yq} \cdot q \frac {\wedge^\bullet \left(\frac {\CU'}{yq} \right)}{\wedge^\bullet \left( \frac {\CU m}{yq} \right)} - \Upsilon_{n-1, \frac {yq}m}' \cdot q \gamma^{-n+1}
$$
This equality of classes on $\CM_c \times S \times \CM_{c'}$ precisely underlies equality \eqref{eqn:two}.

\end{proof}

\subsection{}

In all formulas below, whenever one encounters a product of several $L$, $E$, $U$ operators, one needs to place the symbol $|_\Delta$ next to it, e.g. $L(\dots)E(\dots)U(\dots)|_\Delta$ as in \eqref{eqn:op w}. From now on, we will suppress the notation $|_\Delta$ from our formulas for brevity. \\

\begin{proof} \emph{of Theorem \ref{thm:main}:} In terms of the generating series \eqref{eqn:big series}, formulas \eqref{eqn:one} and \eqref{eqn:two} take the following form:
\begin{align*}
&\left(1 - x \right) A_m L(x,y)  = L \left( x \gamma , \frac ym \right) A_m - x L \left( x \gamma , \frac {yq}m \right) E \left(\frac {yq}m \right) A_m E(y)^{-1} \\
&U\left( x \gamma, \frac {yq}m \right) A_m \left(1 - \frac qx \right) = A_m U(x,y) - \frac qx E \left( \frac {yq}m \right)^{-1} A_m E(yq) U(x,yq)
\end{align*}
Change the variables $x \mapsto xq$, $y \mapsto y/q$ in the second equation, and multiply the first equation by $E(y)$ and the second equation by $E(y/m)$. Thus we obtain:
\begin{align*}
&\left(1 - x \right) A_m L(x,y) E(y) = \\
& \qquad \qquad = L \left( x \gamma ,  \frac ym \right) A_m E(y) - x L \left( x \gamma, \frac {yq}m \right) E \left(\frac {yq}m \right) A_m   \\
&E \left( \frac ym \right) U\left( xq \gamma , \frac ym\right) A_m \left(1 - \frac 1x \right) = \\
& \qquad \qquad = E \left( \frac ym \right) A_m U\left(xq, \frac {y}q \right) - \frac 1x A_m E(y) U\left(  xq,y \right)
\end{align*}
Now let us replace the variable $y$ by the symbol $yD_x$, where $D_x$ denotes the $q$-difference operator $D_x(f(x)) = f(xq)$. However, we make the following prescription: in the first equation above, the $D_x$'s are placed to the right of all $x$'s, while in the second equation, the $D_x$'s are placed to the left of all the $x$'s. We thus obtain:
\begin{align*}
&\left(1 - x \right) A_m L(x,yD_x) E(yD_x) = \\
& \qquad \qquad = L \left( x \gamma ,  \frac {yD_x}m \right) A_m E(yD_x) - x L \left( x \gamma, \frac {yD_xq}m \right) E \left(\frac {yD_xq}m \right) A_m   \\
&E \left( \frac {yD_x}m \right) U\left( xq \gamma , \frac {yD_x} m\right) A_m (1-x) = \\
& \qquad \qquad =  A_m E(yD_x) U\left(  xq,yD_x \right) - E \left( \frac {yD_x}m \right) A_m U\left(xq, \frac {yD_x}q \right) x
\end{align*}
Now let us multiply the first equation on the right by $U(qx,yD_x)$ (with the $D_x$'s placed to the left of all the $x$'s) and the second equation on the left by $L(x \gamma,yD_x/m)$ (with the $D_x$'s placed to the right of all the $x$'s):
\begin{align*}
&(1-x) A_m L(x,yD_x) E(yD_x) U(xq,yD_x) = \\
& \quad = L \left( x\gamma ,  \frac {yD_x}m \right) A_m E(yD_x) U(xq,yD_x) - x L \left( x \gamma, \frac {yD_xq}m \right) E \left(\frac {yD_xq}m \right) A_m U(xq,yD_x)  \\
&L \left(x \gamma,\frac {yD_x}m \right) E \left( \frac {yD_x}m \right) U\left( xq \gamma , \frac {yD_x} m\right) A_m (1-x) = \\
& \quad = L \left(x \gamma,\frac {yD_x}m \right) A_m E(yD_x) U\left(  xq,yD_x \right) - L \left(x \gamma,\frac {yD_x}m \right) E \left( \frac {yD_x}m \right) A_m U\left(xq, \frac {yD_x}q \right) x
\end{align*}
The two terms in the right-hand sides of the above equations are pairwise equal to each other (this is not manifestly obvious for the second term, because $y$ differs from $yq$, but this is a consequence of commuting $D_x$ past $x$). We conclude that:
\begin{multline*}
(1-x)A_m L(x,yD_x) E(yD_x) U(xq,yD_x) = \\ = L \left(x \gamma,\frac {yD_x}m \right) E \left( \frac {yD_x}m \right) U\left( xq \gamma , \frac {yD_x} m\right) A_m (1-x)
\end{multline*} 
Recalling the definition \eqref{eqn:op w series}, this implies 
$$
(1-x) A_m W(x,yD_x)  = W \left(x\gamma,\frac {yD_x}m \right) A_m (1-x) 
$$
Taking the coefficient of $(yD_x)^{-k}$ implies \eqref{eqn:comm a}. In doing so, the right-most factor $1-x$ changes into $1 - \frac x{q^k}$ due to the fact that the operators $\frac 1{D_x^{k}}$ must pass over it.

\end{proof}

\subsection{} Finally, we recall the operator $\Phi_m : \kmm \rightarrow \km$ defined in \eqref{eqn:phi}:
$$
\Phi_m = A_m \exp \left[\sum_{n=1}^\infty \frac {P_n}n \left\{ \frac {(q^n-1)q^{-rn}}{[n]_{q_1} [n]_{q_2}} \right\} \right]
$$
and let us translate \eqref{eqn:comm a}, \eqref{eqn:comm a heis 1}, \eqref{eqn:comm a heis 2} into commutation relations involving $\Phi_m$. \\

\begin{proof} \emph{of Corollary \ref{cor:main}:} Because $P_n$ commutes with $P_{n'}$ for all $n,n'>0$, \eqref{eqn:comm a heis 2} $\Rightarrow$ \eqref{eqn:comm phi heis} when the sign is $+$. Let us now prove \eqref{eqn:comm phi heis} when the sign is $-$. We write:
$$
\Phi_m = A_m \cdot \exp 
$$
where $\exp$ is shorthand for $\exp \left[\sum_{n=1}^\infty \frac {P_n}n \left\{ \frac {(q^n-1)q^{-rn}}{[n]_{q_1} [n]_{q_2}} \right\} \right]$. Then \eqref{eqn:comm a heis 1} reads:
$$
\Phi_m \cdot \exp^{-1} \cdot P_{-n} - P_{-n} \cdot \Phi_m \cdot \exp^{-1} \gamma^n = \Phi_m \cdot \exp^{-1}  (1 - \gamma^n)
$$
The relation above will establish \eqref{eqn:comm phi heis} for $\pm = -$ once we prove that:
\begin{equation}
\label{eqn:rus 0}
[\exp^{-1}, P_{-n}] = (1-q^{-rn}) \exp^{-1}
\end{equation}
If we take the logarithm of \eqref{eqn:rus 0}, it boils down to:
\begin{equation}
\label{eqn:rus}
\left[P_{-n}, \frac {P_n}n \left\{ \frac {(q^n-1)q^{-nr}}{[n]_{q_1} [n]_{q_2}} \right\} \right] = 1-q^{-rn}
\end{equation}
Relation \eqref{eqn:rus} is an equality of operators $\km \rightarrow \kms$ (the right-hand side denotes pull-back multiplied by $\proj_S^*(1-q^{-rn})$), and it is proved as follows. Take equality \eqref{eqn:q heis} of operators $\km \rightarrow \kmss$, multiply it by:
\begin{equation}
\label{eqn:class 3}
\text{proj}_3^* \left( \frac 1n \cdot\frac {(q^n-1)q^{-nr}}{[n]_{q_1} [n]_{q_2}}\right) \in K_{\CM \times S \times S}
\end{equation}
and then apply $\text{proj}_{12*}$ to the result (above, we write $\CM \times S \times S \xrightarrow{\text{proj}_{12}, \text{proj}_3} \CM \times S, S$ for the obvious projection maps). The outcome of this procedure is precisely \eqref{eqn:rus}. \\

\noindent Now let us prove \eqref{eqn:comm a} $\Rightarrow$ \eqref{eqn:comm phi}. To do so, we must take formula \eqref{eqn:w to p series} for $\pm = +$ (which is a priori an equality of operators $\km \rightarrow \kmss$), multiply it by \eqref{eqn:class 3} and then apply $\text{proj}_{12*}$ to the result. The resulting equality reads:
$$
\left[ W_k(x), \frac {P_n}n \left\{ \frac {(q^n-1)q^{-nr}}{[n]_{q_1} [n]_{q_2}} \right\} \right] = \frac {(1 - q^{-nk})x^n}n  W_k(x) 
$$
It is an easy exercise the show that $[W,P] = cW$ implies that $\exp(-P)W = \exp(c) \cdot W\exp(-P)$ as long as $c$ commutes with both $W$ and $P$. Therefore, we infer that:
\begin{align*} 
&\exp^{-1} W_k(x) = \exp\left[ \sum_{n=1}^{\infty} \frac {(1 - q^{-nk})x^n}n \right] W_k(x) \exp^{-1}  \Rightarrow \\
&\Rightarrow \ \exp^{-1} W_k(x) = \frac {1-\frac x{q^k}}{1-x} \cdot W_k(x) \exp^{-1} \Rightarrow \\
&\Rightarrow \ \Phi_m \exp^{-1} W_k(x) \cdot (1-x) = \Phi_m W_k(x) \exp^{-1} \cdot \left(1 - \frac x{q^k} \right)
\end{align*}
With this in mind, \eqref{eqn:comm a} and the fact that $\Phi_m \exp^{-1} = A_m$ imply that:
$$
m^kW_k(x\gamma) \Phi_m \exp^{-1} \cdot \left(1- \frac x{q^k} \right) = \Phi_m  W_k(x) \exp^{-1} \cdot \left(1- \frac x{q^k} \right) 
$$
Multiplying on the right with $\exp$ yields \eqref{eqn:comm phi}.

\end{proof}

\section{The Verma module}
\label{sec:verma}

\medskip

\subsection{}
\label{sub:verma intro}

Let us now specialize to $S = \BA^2$, and explain all the necessary modifications to the constructions in the present paper (we refer the reader to \cite[Section 3]{W} for details). From here on, let $\CM$ be the moduli space parameterizing rank $r$ torsion-free sheaves $\CF$ on $\BP^2$, together with a trivialization along a fixed line $\infty \subset \BP^2$:
$$
\CM = \Big\{\CF, \CF|_\infty \stackrel{\phi}\cong \CO_\infty^r \Big\}
$$
The $c_1$ of such sheaves is forced to be 0, but $c_2$ is free to vary over the non-negative integers, and so the moduli space breaks up into connected components as before: 
$$
\CM = \bigsqcup_{c = 0}^\infty \CM_c
$$
The space $\CM$ is acted on by the torus $T = \BC^* \times \BC^* \times (\BC^*)^r$, where the first two factors act by scaling $\BA^2$, and the latter $r$ factors act on the framing $\phi$. Note that:
$$
K^T_0(\pt) = \BZ[q_1^{\pm 1}, q_2^{\pm 1}, u_1^{\pm 1},\dots,u_r^{\pm 1}]
$$
where $q_1,q_2,u_1,\dots,u_r$ are the standard elementary characters of the torus $T$. We note that $q_1$ and $q_2$ are the equivariant weights of $\Omega_{\BA^2}^1$, and the determinant of the universal sheaf $\CU$ is the equivariant constant $u = u_1\dots u_r$. Consider the group:
$$
\km = \bigoplus_{c=0}^\infty K^T_0(\CM_c) \underset {\BZ[q_1^{\pm 1}, q_2^{\pm 1},u_1^{\pm 1},\dots,u_r^{\pm 1}]}{\otimes} \BQ(q_1,q_2,u_1,\dots,u_r)
$$
The main goal of \loccit was to define operators akin to \eqref{eqn:op w}, \eqref{eqn:op p 1} and \eqref{eqn:op p 2}:
\begin{equation}
\label{eqn:ops}
W_{n,k}, P_{\pm n'} : \km \rightarrow  \km
\end{equation}
for all $n \in \BZ$ and $k,n' \in \BN$, and then show that these operators satisfy the relations in the deformed $W$--algebra of type $\fgl_r$ (since $S = \BA^2$, $\km \cong \kms$ naturally). \\ 

\begin{definition}

(\cite[Definition 2.28]{W}) Let $q_1,q_2,u_1,\dots,u_r$ be formal symbols. The universal Verma module $M_{u_1,\dots,u_r}$ is the $\BQ(q_1,q_2,u_1,\dots,u_r)$-vector space with basis given by:
\begin{equation}
\label{eqn:basis}
W_{n_1,k_1} \dots W_{n_s,k_s} \vac
\end{equation}
as the pairs $(n_i,k_i)$ range over $-\BN \times \{1,\dots,r\}$ and are ordered in non-decreasing order of the slope $n_i/k_i$. We make $M_{u_1,\dots,u_r}$ into a deformed $W$--algebra module as follows. The action of an arbitrary generator $W_{n,k}$ on the basis vector \eqref{eqn:basis} is prescribed by the commutation relations \eqref{eqn:comm w}, together with the relations:
\begin{align*}
&W_{n,k} \vac = 0 & &\text{if } n > 0 \text{ or } k>r \\
&W_{0,k} \vac = e_k(u_1,\dots,u_r)\vac & &\text{for all }k
\end{align*} 
where $e_k$ denotes the $k$--th elementary symmetric polynomial. \\

\end{definition}

\begin{theorem} \label{thm:fock} (\cite[Theorem 3.12]{W}) We have an isomorphism of modules for the deformed $W$-algebra of type $\fgl_r$ (the action in the LHS is given by \eqref{eqn:ops}):
\begin{equation}
\label{eqn:iso fock}
\km \cong M_{u_1,\dots,u_r}
\end{equation}
induced by sending the $K$-theory class of the structure sheaf of $\CM_0 \subset \CM$ to $|\varnothing\rangle$. \\

\end{theorem}

\subsection{}
\label{sub:weak} 

The Ext (respectively vertex) operator $A_m$ (respectively $\Phi_m$) for $S = \BA^2$ was studied in \cite[Section 4]{W}, where we obtained an analogue of Theorem \ref{thm:main} in the case $k=1$ (some coefficients in the formulas of \loccit differ from those of the present paper, because their operator $A_m$ differs from ours by an equivariant constant). However, having only proved the case $k=1$ in \loccit led to weaker formulas than \eqref{eqn:comm a}. Thus, the present paper strengthens the results of \loccit (see Remark 4.8 therein). Specifically, Corollary \ref{cor:main} completely determines the operator $\Phi_m$ (hence also $A_m$) in the case $S = \BA^2$, due to Theorems \ref{thm:fock} and \ref{thm:unique}. \\

\begin{theorem} 
\label{thm:unique}

Given two Verma modules $M_{u_1,\dots,u_r}$ and $M_{u'_1,\dots,u'_r}$, there is a unique (up to constant multiple in $\BQ(q_1,q_2,u_1,\dots,u_r,u_1',\dots,u_r')$) linear map: 
$$
\Phi_m : M_{u_1',\dots,u_r'} \rightarrow M_{u_1,\dots,u_r}
$$
satisfying \eqref{eqn:comm phi} for all $k \geq 1$. \\

\end{theorem}

\begin{proof} The existence of such a linear map follows from the very fact that the operator \eqref{eqn:phi} satisfies \eqref{eqn:comm phi}. To show uniqueness, it is enough to prove $\langle \varnothing | \Phi_m | \varnothing \rangle = 0$ implies $\Phi_m = 0$, for any operator that satisfies the following relations for all $n$, $k$:
\begin{equation}
\label{eqn:comm phi foil}
\Phi_m W_{n,k} - \Phi_m W_{n+1,k} \cdot q^{-k} = W_{n,k} \Phi_m \cdot m^k \gamma^{-nk} - W_{n+1,k} \Phi_m \cdot \frac {m^k}{q^k} \gamma^{-(n+1)k}
\end{equation}
where $m$ and $\gamma$ are certain non-zero constants. \\

\begin{claim}
\label{claim:shapovalov}

For any parameters $u_1,\dots,u_r$, there exists a non-degenerate pairing:
$$
M_{u_1,\dots,u_r} \otimes M_{u_1,\dots,u_r} \xrightarrow{\langle \cdot, \cdot \rangle} \BQ(q_1,q_2,u_1,\dots,u_r)
$$
such that the adjoint of $W_{n,k}$ is $W_{-n,k}$, for all $n \in \BZ$ and $k \in  \BN$. \\

\end{claim}

\begin{proof} Using Theorem \ref{thm:fock}, the required pairing is provided by the equivariant Euler characteristic pairing on $K_{\CM}$ (renormalized as in \cite[Section 3.14]{W}). The operators $W_{n,k}$ and $W_{-n,k}$ are adjoint with respect to this pairing (\cite[formula (3.39)]{W}). 

\end{proof}

\noindent Let us now complete the proof of Theorem \ref{thm:unique}. Because Verma modules are generated by $W_{n,k}$ acting on $\varnothing$, then we must show that $\langle \varnothing |\Phi_m|\varnothing \rangle = 0$ implies:
\begin{equation}
\label{eqn:adam}
\langle \varnothing |W_{-n_s,k_s}\dots W_{-n_1,k_1} \Phi_m W_{n_1',k_1'}\dots W_{n_t',k_t'}| \varnothing \rangle = 0
\end{equation}
for all collections of indices $(n_i,k_i), (n_i',k_i') \in \BZ_{\leq 0} \times \{1,\dots,r\}$, ordered by slope:
$$
\frac {n_1}{k_1} \leq \dots \leq \frac {n_s}{k_s}, \qquad \frac {n_1'}{k_1'} \leq \dots \leq \frac {n_t'}{k_t'}
$$
The matrix coefficient \eqref{eqn:adam} is non-zero only if the $n_i$'s and $n'_j$'s are all non-positive, so we will prove formula \eqref{eqn:adam} by induction on the non-positive integer $\delta = \sum n_i + \sum n_i'$. We may assume that $n_s, n_t' < 0$ because $W_{0,k}|\varnothing\rangle$ is a multiple of $|\varnothing\rangle$ for any $k$. The base case $\delta = 0$ of the induction is simply the assumption $\langle \varnothing |\Phi_m|\varnothing \rangle = 0$. As for the induction step, let us iterate relation \eqref{eqn:comm phi foil} to obtain:
$$
\Phi_m W_{n_1',k_1'} \dots W_{n_t',k_t'}\in \text{span} \begin{cases} \Phi_m W_{n_1'+\varepsilon_1,k_1'} \dots W_{n_t'+\varepsilon_t,k_t'} \\ W_{n_1'+\varepsilon'_1,k_1'} \dots W_{n_t'+\varepsilon'_t,k_t'} \Phi_m \end{cases}
$$
where $\varepsilon_1,\dots,\varepsilon_t \in \{0,1\}$ are not all 0, and $\varepsilon_1',\dots,\varepsilon_t' \in \{0,1\}$. That means that the left-hand side of \eqref{eqn:adam} is a linear combination of:
$$
\langle \varnothing |W_{-n_s,k_s}\dots W_{-n_1,k_1} \Phi_m W_{n_1'+\varepsilon_1,k_1'}\dots W_{n_t'+\varepsilon_t,k_t'}| \varnothing \rangle 
$$
(which is 0 by the induction hypothesis, because the $\varepsilon_i$'s are not all 0) and: 
\begin{equation}
\label{eqn:smith}
\langle \varnothing |W_{-n_s,k_s}\dots W_{-n_1,k_1} W_{n_1'+\varepsilon'_1,k_1'} \dots W_{n_t'+\varepsilon'_t,k_t'} \Phi_m| \varnothing \rangle 
\end{equation}
The induction step will be complete once we show that \eqref{eqn:smith} is 0. As a consequence of \eqref{eqn:comm w}, the product of $W$'s in \eqref{eqn:smith} can be written as a linear combination of:
$$
W_{-n_r'',k_r''}\dots W_{-n_1'',k_1''} \quad \text{with} \quad \frac {n_1''}{k_1''} \leq \dots \leq \frac {n_r''}{k_r''}
$$
and $\sum n_i'' = \sum n_i - \sum n_i' - \sum \varepsilon_i'$ for degree reasons. If $n_r'' > 0$, then the product of $W$'s above annihilates $\langle \varnothing|$. Thus, we may assume $n_r'' \leq 0$, in which case the fact that: 
$$
\sum n_i'' = \sum n_i - \sum n_i' - \sum \varepsilon_i' > \sum n_i + \sum n_i'
$$
(recall that $n_i' < 0$ by assumption, while $\varepsilon_i' \in \{0,1\}$) means that we can apply the induction hypothesis to conclude that \eqref{eqn:smith} is 0.

\end{proof}

\noindent We note that the identification of $A_m$ (in the case $S = \BA^2$) with a vertex operator was also achieved in \cite{BFMZZ}, who computed relations \eqref{eqn:one} and \eqref{eqn:two} for $n=1$ in the basis of fixed points. This uniquely determines the operator $A_m$ due to certain features of the Ding-Iohara-Miki algebra, but does not directly establish the connection with the generating currents of the deformed $W$--algebra of $\fgl_r$. From a geometric point of view, this is because the Nakajima-type simple correspondences only describe the operators $L_{1,k}$ and $U_{1,k}$. As we have seen in Subsection \ref{sub:basic mod}, in order to define the operators $L_{n,k}$ and $U_{n,k}$ for all $n$ (with the ultimate goal of defining the $W$--algebra generators $W_{n,k}$ in \eqref{eqn:op w}), one needs to introduce the more complicated correspondences \eqref{eqn:zn}.

\end{document}